\documentclass[a4paper,12pt, reqno]{amsart}

\usepackage{geometry}
\usepackage[bbgreekl]{mathbbol}
\usepackage[new]{old-arrows}

\usepackage{amsmath, amssymb, amsthm, url}
\usepackage{euscript}
\usepackage{enumitem}
\usepackage{tikz-cd}
\usepackage{quiver}
\usepackage{mathrsfs}
\usepackage{mathtools}
\usepackage{graphicx}
\usepackage{thmtools}

\usepackage{multirow}
\usepackage{array}
\def\checkmark{\tikz\fill[scale=0.4](0,.35) -- (.25,0) -- (1,.7) -- (.25,.15) -- cycle; }

\usepackage{stmaryrd}
\SetSymbolFont{stmry}{bold}{U}{stmry}{m}{n}

\usepackage[backend=biber,style=alphabetic]{biblatex}
\usepackage[simplepostnote]{lumsdaine-biblatex-misc}
\renewbibmacro{in:}{}

\usepackage[svgnames]{xcolor}
\definecolor{darkgreen}{rgb}{0,0.45,0}
\definecolor{darkred}{rgb}{0.75,0,0}
\definecolor{darkblue}{rgb}{0,0,0.6}
\usepackage[colorlinks,citecolor=darkgreen,linkcolor=darkblue,urlcolor=darkblue]{hyperref}
\usepackage{cleveref}
\usepackage{microtype}

\usepackage{soul}
\setul{3pt}{.4pt}
\usepackage{dsfont}

\usetikzlibrary{arrows.meta}

\usetikzlibrary{calc}
\usetikzlibrary{fit}
\usetikzlibrary{shapes}
\usetikzlibrary{arrows}
\usetikzlibrary{decorations.markings}
\usetikzlibrary{decorations.pathmorphing}
\usetikzlibrary{positioning}
\usetikzlibrary{intersections}

\tikzset{
  commutative diagrams/arrow style=tikz,
}

\tikzset{cd-style/.style={commutative diagrams/every diagram}}
\tikzset{cd-arrow-style/.style={commutative diagrams/.cd, every arrow, every label}}

\makeatletter
\tikzset{
  label/.style n args={2}{
    edge node={node [
      execute at begin node=\iftikzcd@mathmode$\fi,
      execute at end node=\iftikzcd@mathmode$\fi,
      /tikz/commutative diagrams/.cd,every label,
      #2
      ] {#1}}
  }
}
\makeatother

\tikzset{fibtip/.tip={Triangle[open,angle=60:4.5pt]}}
\tikzset{tfibtip/.tip={Bar[sep]Triangle[open,angle=45:4pt]}} 

\pgfarrowsdeclaredouble{fibbtip}{fibbtip}{fibtip}{.fibtip}

\tikzset{ulttip/.tip={Computer Modern Rightarrow[reversed]}}

\pgfarrowsdeclarealias{c}{c}{left hook}{right hook}
\pgfarrowsdeclarealias{c'}{c'}{right hook}{left hook}

\tikzset{fib/.code={\pgfsetarrowsend{fibtip}}}
\tikzset{fibb/.code={\pgfsetarrowsend{fibbtip}}}
\tikzset{tfib/.code={\pgfsetarrowsend{tfibtip}}}
\tikzset{inj/.code={\pgfsetarrowsstart{c}}}
\tikzset{inj'/.code={\pgfsetarrowsstart{c'}}}
\tikzset{cover/.style={->>}}
\tikzset{tcof/.style={tail}}
\tikzset{zigzag/.style={commutative diagrams/rightsquigarrow}}
\tikzset{ulttip/.code={\pgfsetarrowsend{ulttip}}}

\tikzset{suplabel/.style={label={#1}{auto=left,pos=1}}}
\tikzset{sublabel/.style={label={#1}{auto=right,pos=1}}}
\tikzset{sub/.style={label={#1}{auto=right,pos=1}}}

\tikzset{lw/.style={"lw"{#1}},lw/.default={very near end}}
\tikzset{lw'/.style={"lw"'{#1}},lw'/.default={very near end}}
\tikzset{weq/.style={"\sim"}}
\tikzset{weq'/.style={"\sim"'}}
\tikzset{ult/.style={ulttip,"#1"'{very near end}}}
\tikzset{ult'/.style={ulttip,"#1"{very near end}}}

\newcommand{\generalto}[2]{ \mathrel{\mkern-1mu
  \tikz[baseline={([yshift=-0.58ex]a.south)}]{%
    \node[minimum width=1em,align=center,inner xsep=0.5ex,inner ysep=0.25ex] (a) {$\scriptstyle #2$};
    \draw[cd-arrow-style,#1] (a.south west) -- (a.south east);}
 \mkern-1mu}}

\renewcommand{\to}[1][]{ \generalto{->}{#1} }

\newcommand{\injto}[1][]{ \generalto{->,inj}{#1} }

\newcommand{\coverto}[1][]{ \generalto{->>}{#1\ \,} }
\newcommand{\vuto}[2][]{ \mathrel{{\generalto{ulttip}{#1}}\mkern2mu_{{#2}}\mkern-2mu}}
\newcommand{\vuleq}[2][]{ \mathrel{\preccurlyeq_{{#2}}^{{#1}}\mkern-2mu}}

\setlist[enumerate]{label={(\roman*)}, ref={(\roman*)}} 

\theoremstyle{plain}
\newtheorem{Thm}{Theorem}[section]

\newtheorem{prop}[Thm]{Proposition}
\newtheorem{cor}[Thm]{Corollary}
\newtheorem{lem}[Thm]{Lemma}

\theoremstyle{definition}
\newtheorem{defn}[Thm]{Definition}

\newtheorem{nota}[Thm]{Notation}

\newtheorem{ex}[Thm]{Example}
\newtheorem{rmk}[Thm]{Remark}
\newtheorem*{nrmk}{Remark}

\newcommand\pref[1]{\ref{#1}}

\DeclareMathOperator{\cl}{\mathsf{cl}}
\DeclareMathOperator{\subcl}{\mathsf{scl}}
 
\newcommand{\mono}{\hookrightarrow}
\newcommand{\epi}{\twoheadrightarrow}

\newcommand{\C}{\mathscr{C}}
\newcommand{\E}{\mathscr{E}}
\newcommand{\F}{\mathscr{F}}
\newcommand{\B}{\mathscr{B}}
\newcommand{\G}{\mathscr{G}}

\newcommand{\D}{\mathscr{D}}
\newcommand{\A}{\mathscr{A}}

\newcommand{\X}{\mathscr{X}}

\newcommand{\Z}{\mathscr{Z}}
\newcommand{\Y}{\mathscr{Y}}

\newcommand{\eM}{\EuScript{M}}

\newcommand{\dsA}{\mathds{A}}
\newcommand{\dsD}{\mathds{D}}

\newcommand{\ptuf}{\star}

\newcommand{\Sh}{\mathsf{Sh}}

\newcommand{\Sub}{\mathsf{Sub}}

\newcommand{\Topos}{\mathsf{Topos}}
\newcommand{\Toposwep}{\mathsf{Topos}_{\mathrm{wep}}}

\newcommand{\op}{\mathrm{op}}

\newcommand{\TopSp}{\mathsf{TopSp}}

\newcommand{\pt}{\mathsf{pt}}

\newcommand{\id}{\mathsf{id}}
\newcommand{\ev}{\mathsf{ev}}
\newcommand{\Set}{\mathsf{Set}}

\newcommand{\SetO}{\mathsf{Set}[\mathbb{O}]}
\newcommand{\UF}{\mathds{U}}

\newcommand{\Loc}{\mathsf{Loc}}

\newcommand{\Open}{\mathcal{O}}
\newcommand{\Ob}{\mathsf{Ob}}
\newcommand{\Hom}{\mathsf{Hom}}

\newcommand{\ie}{i.e., }

\newcommand{\vuCat}{\mathsf{vuCat}}
\newcommand{\vuCatb}{\mathsf{vuCat_{b}}}
\newcommand{\vuCAT}{\mathsf{vuCAT}}
\newcommand{\vuPos}{\mathsf{vuPos}}

\newcommand{\vuPOS}{\mathsf{vuPOS}}
\newcommand{\vuPostaut}{\mathsf{vuPos_{taut}}}

\newcommand{\vuPOStaut}{\mathsf{vuPOS_{taut}}}
\newcommand{\vupt}{\pt}

\newcommand{\Tref}[1]{|#1|} 
\newcommand{\Lref}[1]{|#1|} 
\newcommand{\Tinj}{\iota}
\newcommand{\Linj}{\iota}
\newcommand{\et}[1]{\mathrm{\acute{E}t}_{#1}} 

\DeclareFontFamily{U}{dmjhira}{}
\DeclareFontShape{U}{dmjhira}{m}{n}{ <-> dmjhira }{}

\newcommand{\ptloc}{\pt}
\newcommand{\shloc}{\Open}
\newcommand{\ptvu}{\vupt}
\newcommand{\shvu}{\Sh}

\newcommand{\emaillink}[1]{\href{mailto:#1}{\sf #1}}
\title[Geometric morphisms of vu-categories]{Geometric morphisms of \\ virtual ultracategories}

\newcommand{\squig}{\mathrel{\tikz[baseline=-0.5ex, scale=0.5]{\draw[<->, decorate, decoration={snake, amplitude=0.1mm, segment length=1mm}] (0,0) -- (1,0);}}}
\newcommand{\crossmark}{\boldsymbol\times}
\newcolumntype{C}{>{\centering\arraybackslash}p{30pt}}

\date{\today}

\author{Gabriel Saadia}
\author{Errol Yuksel}

\address{
Gabriel \textsc{Saadia}: \newline
Department of Mathematics\newline
Stockholm University\newline
Stockholm, Sweden\newline
\emaillink{gabrielsaadia@gmail.com}
}

\address{
Errol \textsc{Yuksel}: \newline
Department of Mathematics\newline
Stockholm University\newline
Stockholm, Sweden\newline
\emaillink{errol.yuksel@math.su.se}
}

\thanks{The authors were supported by the KAW Foundation, under the project “Type Theory for Mathematics and Computer Science” (PI Thierry Coquand).}

\begin{document}

\begin{abstract}
    We build on the recent development that toposes with enough points can be represented as categories of points equipped with ultraconvergence data (a.k.a. virtual ultracategories), characterizing usual classes of geometric morphisms --- surjections, embeddings, hyperconnected maps, and localic maps --- in terms of their points. 
    
    In particular, the characterisation of geometric surjections relies on a notion of “virtual ultraretractions”, which generalises usual retractions and filtered colimits, categorifies Johnstone’s notion of subclosure for topological spaces, and provides a new characterisation of separating classes of points of a topos.
    
    

\end{abstract}

\maketitle
\setcounter{tocdepth}{1}\tableofcontents

\section*{Introduction}

Grothendieck toposes are usually presented as a categorification of topological spaces. This is not completely accurate: toposes actually categorify the point-free notion of \emph{locales}. A notion of categorified space should stand in relation to topological spaces the same way that toposes do to locales; in other words, it should be to toposes what topological spaces are to locales. One such notion has been introduced by Garner under the name of \emph{ionad} in \cite{IonadGarner} and generalised by Di Liberti in \cite{IonadIvan}. Recently, a new proposal of such a notion of categorified space has been introduced independently in three different works \cite{Saa,Hamad2,SUJ} dubbed \emph{virtual ultracategories} in the first above-cited work --- we will abbreviate it to vu-categories. 

A topological space on a set of points can be seen as a posetal multi-categorical structure where the objects are given by the points of the space and multi-arrows are given by the convergence of ultrafilters: there is an arrow $a\vuleq{\sigma}b_s$ from a point $a$ to an ultrafamily of points $(b_s)_{s:\sigma}$ precisely if the point is a limit point of this ultrafamily. Vu-categories are a proof-relevant generalisation of the above: ultrafamilies of points can converge to a point in several different ways, corresponding to several \emph{vu-arrows} $a\vuto{\sigma}b_s$. The three above-mentioned papers focus on proving a \emph{reconstruction theorem} for vu-categories, showing that Grothendieck toposes with enough points embed 2-fully faithfully into the 2-category of vu-categories, but leave the broader development of the theory of these new objects for future work.

One of the main insights this novel approach teaches us is that vu-categories, and therefore toposes with enough points, have a multi-categorical nature; one can thus transpose standard categorical notions to the vu-categorical setting.
Since toposes encode categories of categorified opens, \emph{sheaves}, of categorified spaces, geometrical constructions and concepts there need to be expressed in the language of opens and sheaves, even when the point-centric perspective simplifies the intuition and the computations. The reconstruction theorem suggests that it is possible to give full characterisations of geometrical notions pertaining to toposes with enough points in terms of their vu-categories of points. In this paper, we undertake the task of giving such characterisations for some usual classes of geometric morphisms, which we translate into the vu-categorical language.

\subsection*{Notations}
All our toposes will be Grothendieck toposes; we will denote them by $\E,\F,\dots$ and geometric morphisms between them are denoted $f,g,h,\dots$ The 2-category of toposes is denoted by $\Topos$ and the full subcategory of toposes with enough points by $\Toposwep$.
%
%
We denote by $\vuCAT$ the 2-category of vu-categories, by $\vuCatb$ the 2-category of bounded vu-categories (\ie with an accessible category of vu-sheaves), by $\vuPOS$ the 2-category of large vu-posets, and by $\vuPos$ the 2-category of essentially small vu-posets (\ie equivalent to one with a small set of points).

\subsection*{The adjunction}
    
The reconstruction theorem proved in \cite{Saa,Hamad2,SUJ} can be formally presented as an idempotent adjunction between toposes and (bounded) vu-categories with $\Set$ playing the role of the dualising object.
\[\begin{tikzcd}[sep = 80pt]
    \Topos & \vuCatb.
    \arrow[""{name=0, anchor=center, inner sep=0}, "{\pt = \Topos(\Set, -)}"', shift right, from=1-1, to=1-2]
    \arrow[""{name=1, anchor=center, inner sep=0}, "{\Sh = \vuCat(-,\Set)}"'{pos=0.5}, shift right, curve={height=15pt}, from=1-2, to=1-1]
    \arrow["\dashv"{anchor=center, rotate=-91}, draw=none, from=1, to=0]
\end{tikzcd}\]

The fixed points on the $\Topos$ side are exactly toposes with enough points. More precisely, the counit is well understood: $\Sh(\pt(\E))\to\E$ is given by the \emph{restriction of $\E$ to its points}, that is, the largest subtopos of $\E$ having enough points \cite[Def 8.1]{Saa}.

The unit $\eta_\X : \X\to\pt(\Sh(\X))$ is harder to understand. It is known that there exists vu-categories $\X$ that do not arise from points of a topos \cite[5.2]{AliCont}; a toy example of such a vu-category will be given in \Cref{counterex}. In general $\eta_\X$ is neither fully faithful, nor essentially surjective. We first explain the failure of essential subjectivity by introducing the notion of \emph{vu-retractions}.

\subsection*{Vu-retractions}

That $\pt(\Sh(\X))$ can have more points than $\X$ is a 1-dimensional manifestation of the well-known phenomenon that the frame of opens of a topological space can have more points than the space, or put differently, that the sobrification of a topological space make new points appear.

In this paper, we will show that even if the points in $\pt(\Sh(\X))$ are not necessarily in the image of $\eta_\X$, they are \emph{vu-retracts} --- the natural vu-categorical notion of retractions ---  of points in the image of $\eta_\X$. We will say that the unit is \emph{vu-subdense}, that is, is surjective \emph{up to vu-retractions}.

We will actually show that a geometric morphism is surjective if and only if it induces a vu-subdense vu-functor. In particular, a class of objects of $\pt(\E)$ forms the points of a subtopos of $\E$ (that is, corresponds to a geometrically axiomatisable class of models) if and only if it is closed under vu-retractions. Hence, vu-categories enable us to connect together the topological notion of soberness, the categorical notion of retractions, and the logical notion of geometrically axiomatisable class of models.

\begin{nrmk}
    We would like to emphasise this new concept of vu-retraction (\Cref{def:vu-retract}). It is a very flexible notion, powerful enough to recover the sobrification of a topological space and usual retracts, but also the diagonal embedding of an object of an ultracategory into its ultrapower, as well as filtered vu-colimit (\Cref{ex:vu-retract}), which play a central role in the theory of accessible categories.
\end{nrmk}

\subsection*{Having enough sheaves}

That $\eta_\X$ is not necessarily fully faithful is a purely 1-dimensional phenomenon that does not manifest for topological spaces and locales. This phenomenon happens locally, on the homsets, where the \emph{thickening} operation plays an important role.
For a vu-arrow $f: a\vuto{\sigma}b_s$ and for a map of ultrafilters $\phi : \tau\coverto\sigma$ (which is always surjective), we can ‘reindex’ the convergence along $\phi$ and get a vu-arrow $\phi^*f : a\vuto{\tau}b_{\phi t}$ called the \emph{thickening} of $f$ along $\phi$. This operation says that we can ‘repeat’ the points in the categorified convergence $a\vuto{\sigma}b_s$.
We will introduce the new notion of \emph{tautness} for a vu-category (\Cref{def:taut}), as an algebraic condition on the behaviour of the thickening operation: if a vu-arrow behaves like if it were a thickening then it can be uniquely unthicken.

We will show that even if $\eta_\X$ is not full on the nose, it is always \emph{vu-full} --- that is, is full up to thickening.
Hence the unit cannot create genuinely new vu-arrows, but only unthickenings of vu-arrows already in $\X$. Concretely, $\eta_\X$ only identifies parallel vu-arrows together, and tautness forces some unthickening to appear.
Moreover, since $\eta_X$ is vu-full, we can deduce that, for the unit $\eta_\X$ being fully faithful and being faithful are equivalent conditions (assuming $\X$ taut). In that case we will say that $\X$ \emph{has enough sheaves}. The vu-category $\X$ has enough sheaves if we can distinguish different vu-arrows of $\X$ by looking at the vu-sheaves, and this is equivalent to asking for $\X$ to be the vu-category of a class of points of a topos (\Cref{locsobercaract}).

Whereas at the level of points, one should understand $\eta_\X$ as a vu-categorical analogue of the Cauchy-completion, at the level of homsets, one should understand $\eta_\X$ as a ‘connected component’ process, identifying vu-arrows that cannot be distinguished by the vu-category of sets.

We will actually show that any hyperconnected geometric morphism induces a vu-full vu-functor on its points. In the case of the hyperconnected geometric morphism from a topos $\E$ to its localic reflection, we can give this result a logical meaning (\Cref{rmk:Barr-categorified}). For $p$ and $q$ points of $\E$ seen as models of a geometric theory, if any geometric formula satisfied by $p$ is also satisfied by $q$, then there exists a homomorphism of models from $p$ into some ultrapower of $q$ (and an analogous statement can be given for an ultrafamily of models $(q_s)_{s:\sigma}$).

%

\subsection*{Comparing classes of maps}




To show these two results, we will actually take a more general approach and study how classes of geometric morphisms between toposes relate to the corresponding classes of vu-functors between vu-categories. 

We can summarise all the relations we establish in this paper in the following table.
{\itshape
\begin{table}[h]
    \centering
    \begin{tabular}{r l||C C|C C|C C|}
    & & \multicolumn{2}{c|}{$f\squig\pt(f)$}&
        \multicolumn{2}{c|}{$\Sh(F)\squig F$}&
        \multicolumn{2}{c|}{$f\squig F$}\\
    \cline{3-8}
    geometric morphisms & vu-functors & $\Rightarrow$ & $\Leftarrow$ & $\Rightarrow$ & $\Leftarrow$ & $\Rightarrow$ & $\Leftarrow$ \\
    \cline{1-8}
    étale & étale
        & \checkmark & \checkmark 
        & $\crossmark$ & \checkmark
        & $\crossmark$ & \checkmark \\
    embeddings & fully faithful
        & \checkmark & \checkmark 
        & $\crossmark$ & ?
        & $\crossmark$ & \checkmark \\
    localic & faithful
        & \checkmark & \checkmark 
        & $\crossmark$ & ?
        & $\crossmark$ & \checkmark  \\
    surjective & vu-subdense 
        & \checkmark & \checkmark 
        & ? & \checkmark
        & \checkmark & \checkmark \\
    subhyperconnected & vu-full
        & \checkmark & \checkmark
        & $\crossmark$ & \checkmark 
        & \checkmark & \checkmark
    \end{tabular}
    
    \vspace{5pt}\caption{}
    \label{tableau}
\end{table}}
The first column is for $f$ a geometric morphism between toposes \emph{with enough points} and its vu-functor of points $\pt(f)$. The second column is for $F$ a vu-functor between \emph{bounded} vu-categories and the associated geometric morphism on the vu-sheaves $\Sh(F)$. The last column is between $f:\Sh(\X)\to\E$ and $F:\X\to\pt(\E)$ adjunct to each other, with $\X$ bounded and $\E$ having enough points. In each column, $\Rightarrow$ means that if the geometric morphism on the left is in the class then the vu-functor on the right is in the corresponding class, and vice versa for $\Leftarrow$.
This table will guide us throughout the paper; we will fill it line by line (Corollaries \ref{table:etal}, \ref{table:emb}, \ref{table:shc}, \ref{table:loc}, and \ref{table:surj}). 

\begin{nrmk}
    It is worth emphasising the proof technique we use to establish these vu-characterisations: we reduce to the 0-dimensional case by looking at the localic reflection of the slices of the topos, which we call the \emph{localic slices} (\Cref{def:loc-slice}). The crucial point is that subhyperconnected morphisms and geometric surjections can be detected on their localic slices (\Cref{rmk:propclass}). We suspect that this proof technique can be used in other contexts.
\end{nrmk}

\subsection*{Organisation of the paper}

In \Cref{sec:vu-cat} we go over the basics of vu-categories and settle the notations. We introduce the new crucial notion of \emph{tautness} for vu-categories and we show in particular that the vu-categories of points of a topos are taut (\Cref{lemma:taut-sober}). We then recall in \Cref{sec:vu-pos} the relation between vu-posets and topological spaces, and we introduce the \emph{topological reflection} of a vu-category as the vu-categorical counterpart of the usual localic reflection of a topos.

We start our vu-categorical analysis of geometric morphisms in \Cref{sec:ff} by establishing some easy results and filling the first two lines of the table (\ref{tableau}), for étale morphisms and embedding of toposes (\Cref{table:etal}, \Cref{table:emb}). The new technical result needed is that embeddings induce fully faithful vu-functors at the level of vu-categories. We similarly show that localic geometric morphisms induce faithful vu-functors at the level of vu-categories.

To fill the other lines we will need some technical tools that we introduce in \Cref{sec:merging}. First \emph{localic slices} --- and their vu-counterpart \emph{topological étalements} --- will enable us to reduce a 1-categorical problem into multiple 0-dimensional ones, and then \emph{semi-flat diagrams} will enable us to \emph{merge together} several 0-dimensional results shown on the topological étalements to obtain a 1-dimensional result.
\Cref{sec:vu-full} and \Cref{sec:surjectivity} are independent of each other, even though they follow a similar proof pattern: merging the 0-dimensional results shown on the topological étalements along a semi-flat diagram.

\Cref{sec:vu-full} focuses on the relation between vu-full vu-functors and subhyperconnected geometric morphisms. We first study vu-full vu-functors, and we show crucially that, if the codomain is the vu-category of points of a topos, vu-fullness can be checked on the topological étalements (\Cref{shc-vufull-cor}). We then show analogously that subhyperconnectedness can be checked on the localic slices and we settle the relation between vu-full vu-functors and subhyperconnected geometric morphisms (\Cref{table:shc}). We deduce by orthogonality the relation between faithful vu-functors and localic geometric morphisms (\Cref{table:loc}). As a corollary we obtain that vu-categories arising from a class of points of a topos are exactly the bounded taut vu-categories having enough sheaves (\Cref{locsobercaract}).

\Cref{sec:surjectivity} focuses on surjective morphisms, or equivalently on when a class of points is separating. A full answer already exists in the 0-dimensional case: a subset of a topological space is separating exactly if its \emph{subclosure} is the whole space (\Cref{prop:subcl-sep-topological-case}). We then categorify the above (\Cref{prop:subcl-sep-categorified}), one of the main contributions of this paper --- the notion of \emph{vu-retraction} (\Cref{def:vu-retract}) --- emerges from this categorification. Finally, building upon these results, we establish the relation between geometric surjections and \emph{vu-subdense} vu-functors (\Cref{table:loc}).

\subsection*{Acknowledgments}
Both authors would like to thank their PhD advisor \emph{Peter Lumsdaine} for his guidance and his valuable advice during the elaboration of this work.

\section{Crash course in vu-categories}\label{sec:vu-cat}

\subsection{Ultrafilters}

%
Our ultrafilters will be denoted by $\sigma,\tau,\lambda,\kappa,\dots$ respectively defined on sets $S,T,L,K,\dots$ using the variable $s,t,\ell,k,\dots$ for their elements. Morphisms of ultrafilters are taken in the category of ultrafilters $\UF$ as defined in \cite{Blass} and are denoted by $\phi,\psi,\dots$
We will denote the \emph{terminal ultrafilter} by $\ptuf$, the \emph{tensor of ultrafilters} $\sigma\otimes\tau$ by $\sigma\cdot\tau$, and the \emph{sum of ultrafilters} also called \emph{dependent tensor} $\sum_{s:\sigma}\tau_s$ by $\sigma\cdot\tau_s$, the variable $s$ (resp. $t$, $\ell$, $k$) being always associated by the ultrafilter $\sigma$ (resp. $\tau$, $\lambda$, $\kappa$).

For a set $A$ and an ultrafilter $\sigma$, the elements of the \emph{ultrapower} $A^{\sigma}$ are $\sigma$-families of elements of $A$ and are denoted by $(a_s)_{s:\sigma}$ (or $(a_s)_{\sigma}$). We denote by $\int_{s:\sigma}A_s$ (or $\int_{\sigma}A_s$) the \emph{ultraproduct} of the $\sigma$-family of sets $(A_s)_{s:\sigma}$.
For a morphism of ultrafilters $\phi : \tau\to\sigma$, the \emph{reindexing function} $\phi^*:\int_{s:\sigma}A_s\to\int_{t:\tau}A_{\phi t}$ maps $(a_s)_{s:\sigma}$ to $(a_{\phi t})_{t:\tau}$.
We will happily identify $\int_{s:\sigma}A^{\tau_s}$ and $A^{\sum_{s:\sigma}\tau_s}$, \ie we identify $((a_{s,t})_{t:\tau_s})_{s:\sigma}$ with $(a_{s,t})_{(s,t):\sum_{s:\sigma}\tau_s}$, and we denote it by $(a_{s,t})_{s:\sigma\cdot t:\tau_s}$ (or $(a_{s,t})_{\sigma\cdot\tau_s}$).
For a proposition $\varphi$ on the set underlying an ultrafilter $\sigma$, the \emph{ultraquantification} $(\forall s:\sigma)\ \varphi(s)$ means that the set of $s$ satisfying $\varphi$ is $\sigma$-large, and we say that \emph{$\varphi(s)$ holds for all $s:\sigma$} or \emph{holds for $\sigma$-all $s$}. We refer to \cite[Section 1]{Saa} for more details.

Ultrafilters are non-constructive by nature; the only way to construct them is to invoke the \emph{ultrafilter principle} (a non-constructive principle slightly weaker than the usual axiom of choice \cite{HalpernLevy1971}) to extend a proper filter into an ultrafilter. We will use it in the following form.

\begin{lem}\label{lem:extend-uf}
    Let $\alpha\subseteq 2^S$ be a collection of subsets of a set $S$. Then $\alpha$ can be extended to an ultrafilter, \ie there exists an ultrafilter $\sigma$ on $S$ such that $(\forall s:\sigma) \ s\in A$ for all $A\in\alpha$, if and only if all the finite intersections of subsets of $S$ in $\alpha$ are non-empty.
\end{lem}


We fix the following terminology.
\begin{defn}\label{def:final-uf}
    An ultrafilter $\lambda$ on a poset $(L,\leq)$ is said \emph{initial} if it satisfies 
    $(\forall \ell:\lambda)\ \ell\leq \ell_0$
    for each $\ell_0\in L$.
    Analogously, $\lambda$ is said \emph{final} if it satisfies $(\forall \ell:\lambda)\ \ell\geq \ell_0$ for each $\ell_0\in L$.
\end{defn}

We can easily deduce the following from \Cref{lem:extend-uf}.
\begin{lem}\label{lem:ext-final-uf}
    A poset $(L,\leq)$ possesses an initial ultrafilter if and only if it is codirected, \ie for any $p,q\in L$ there exists $r\in L$ such that $r\leq p,q$. An analogous proposition is true for final ultrafilters and directed posets.
\end{lem}

\subsection{Vu-categories}
\begin{defn}\label{def:vucat}
    A \emph{virtual ultracategory} $\X$, or \emph{vu-category} for short, is given by the following data.
    \begin{enumerate}
        \item A collection of objects, also called points.
        \item For every point $a$ and $\sigma$-family of points $(b_s)_{s:\sigma}$, a \emph{homset} $\Hom^\X_{\sigma}(a,b_s)$ (or $\Hom_{\sigma}(a,b_s)$). Elements $f\in\Hom_\sigma(a,b_s)$ are called \emph{vu-arrows of shape} $\sigma$ or \emph{$\sigma$-arrows} and are denoted $f : a \vuto{s:\sigma} b_s$ or simply
        \[f: a \vuto{\sigma}b_s.\]
        
        
        \item For every point $a$, an \emph{identity} $\ptuf$-arrow $\id_a : a\vuto{\ptuf}a$. For a vu-arrow $f:a\vuto{\sigma}b_s$ of shape $\sigma$ and a $\sigma$-family of vu-arrows $(g_s : b_s \vuto{\tau_s}c_{s,t})_{s:\sigma}$, a $\sigma$-\emph{composite}
        \[g_s\circ_{\sigma} f : a\vuto{\sigma\cdot\tau_s}c_{s,t}\]
        that we denote diagrammatically by
        \[a\vuto[f]{\sigma}(b_s\vuto[g_s]{\tau_s}c_{s,t}).\]
        \item\label{diag} For any morphism of ultrafilters $\phi : \tau\to\sigma$ a \emph{reindexing} operation, given by set-maps
        \begin{align*}
            \Hom_\sigma(a,b_s) & \to\Hom_\tau(a,b_{\phi t}) \\
            f & \mapsto \phi^*f
        \end{align*}
        for each $a$ and $(b_s)_{s:\sigma}$.
        The vu-arrow $\phi^*f$ is also called the \emph{thickening} of $f$ along $\phi$. In particular, the thickening of $\id_a$ along the unique map $\sigma \to \ptuf$ is called a \emph{diagonal} of $a$ and is denoted $\delta : a\vuto{\sigma} a$.
    \end{enumerate}
    The above data is required to satisfy the following properties.
    \begin{enumerate}
        \item The identity is neutral and the composition is associative: 
        \[f\circ_{\ptuf} \id_a = f\]
        \[\id_{b_s}\circ_{\sigma} f = f\]
        \[(h_{s,t}\circ_{\tau_s} g_s)\circ_{\sigma} f = h_{s,t}\circ_{\sigma\cdot \tau_s} (g_s\circ_{\sigma} f).\] 
        
        \item The thickening operation is functorial:
        \[{\id_{\sigma}}^*f = f\]
        \[\psi^*(\phi^*f) = (\phi\circ\psi)^*f\]
        that is, it induces functors
        \begin{equation}\tag{$\star$} \label{reindex-functor}
        \begin{split}
            (\UF/\sigma)^{\op} &\to\Set \\
            (\phi : \tau\to\sigma) &\mapsto \Hom_\tau(a,b_{\phi t}). 
        \end{split}
        \end{equation}
        for each $a$ and $(b_s)_{s:\sigma}$.
        \item The thickening operation is compatible with the composition: for $f : a\vuto{\sigma} b_s$ and $(g_s : b_s\vuto{\tau_s} c_{s,t})_{s:\sigma}$
        \begin{enumerate}[label=(\alph*)]
            \item for any $\phi : \lambda\to\sigma$ 
            \[g_{\phi \ell}\circ_{\lambda} (\phi^*f) = (\phi\cdot\id_{\tau_s})^*(g_s\circ_{\sigma} f)\]
            where $\phi\cdot\id_{\tau_s} : \lambda\cdot\tau_{\phi \ell}\to\sigma\cdot\tau_s$,
            \item and for any $(\psi_s : \kappa_s\to\tau_s)_{s:\sigma}$
            \[({\psi_s}^*g_s)\circ_{\sigma} f = (\id_{\sigma}\cdot\psi_s)^*(g_s\circ_{\sigma} f).\]
            where $\id_{\sigma}\cdot\psi_s : \sigma\cdot\kappa_s\to\sigma\cdot\tau_s$.
        \end{enumerate}
    \end{enumerate}
\end{defn}

\begin{rmk}
    We opt for a prefix notation for the thickening rather than a suffix one (as in \cite[3.8]{SUJ}) because we think of thickening as post-composing with a ‘virtual diagonal’. For instance, in the vu-category $\Set$, to thicken along $\phi : \sigma \to \tau$ is to post-compose with the (generalised) diagonal $A^\tau \hookrightarrow A^\sigma$.
\end{rmk}

We now introduce the crucial notion of \emph{tautness} of a vu-category. The behaviour of the thickening operation \eqref{reindex-functor} can be very wild; tautness remedies this issue.

\begin{defn}\label{def:taut}
    A vu-arrow $f:a\vuto{\tau}b_{\phi t}$ is said \emph{reducible along $\phi : \tau \to \sigma$} if $\psi_1^*f = \psi_2^*f$ for all $\psi_1,\psi_2 : \lambda\rightrightarrows \tau$ equalising $\phi$.
    A vu-category $\X$ is said \emph{taut} if any $f:a\vuto{\tau}b_{\phi t}$ reducible along $\phi$ can uniquely be unthickened along $\phi$, \ie there exists a unique vu-arrow $g:a\vuto{\sigma}b_s$ such that $f=\phi^*g$.
\end{defn}

\begin{rmk}
    It is worth noting that this condition can be seen as a sheaf condition: tautness means that for any $a$ and $(b_s)_{s:\sigma}$ the above functor \eqref{reindex-functor}
    is a sheaf over $\UF/\sigma$ for the atomic topology.
    This topology has already been considered in \cite{Eliasson} where it is related to a topology over the category of filters. This opens the way to a study of a filter version of vu-categories \cite{Filterization}.
\end{rmk}

When adapting notions from category theory to vu-categories, we add the prefix vu- to them. For example, we can define vu-functors and vu-natural transformations (see \cite[4.3]{Saa}), that endow the collection of vu-categories with a structure of a (strict) 2-category that we denote by $\vuCAT$.
A vu-functor $F : \X \to \Y$ is said \emph{fully faithful} (resp. faithful) if it induces bijections (resp. injections) on the hom-sets, \ie the set-functions
\[ \Hom^\X_\sigma(a, b_s) \to \Hom^\Y_\sigma(F(a), F(b_s)) \]
are bijective (resp. injective) for all $a$ and $(b_s)_{s:\sigma}$.

\subsection{Points of a topos}

The main way to obtain a vu-category is by taking the points of a topos.
\begin{defn}\label{def:vucat-of-points}
    We denote by $\pt(\E)$ \emph{the vu-category of points} of a topos $\E$. Its objects are given by points of $\E$, \ie geometric morphisms $p : \Set\to\E$, and vu-arrows from $p$ to $(q_s)_{s:\sigma}$ are given by natural transformations $p^*\Rightarrow\int_{s:\sigma}({q_s}^*)$, \ie families of maps $(p^*E\to\int_{s:\sigma}{q_s}^*E)$ natural in $E\in\E$.
\end{defn}

This induces a 2-functor 
\[\pt : \Topos \to \vuCAT,\]
we refer to \cite[4.2(iii)]{Saa} (see also \cite[5.2]{Hamad2} \cite[3.16]{SUJ}) for more details about this construction. 
In particular, by taking the topos $\SetO$, we get a vu-category of small sets, $\Set = \pt(\SetO)$, vu-arrows from a set $A$ to a $\sigma$-family of sets $(B_s)_{s:\sigma}$ are given by set-functions $A\to\int_{s:\sigma}B_s$.

We now show that vu-categories arising as points of a topos are taut; this essentially follows from the following set-theoretical lemma about ultraproducts.
\begin{lem}\label{lemma:taut-sober}
    Let $(A_s)_{s:\sigma}$ be a $\sigma$-family of sets and $\phi : \tau\to\sigma$. The reindexing
    \[\phi^* : \textstyle\int_{s:\sigma}A_s\hookrightarrow\int_{t:\tau}A_{\phi t}\]
    is injective, and its image is described by the $x\in \int_{t:\tau}A_{\phi t}$ such that $\psi_1^*x = \psi_2^*x$ for all pairs of maps $\psi_1,\psi_2 : \lambda\rightrightarrows \tau$ equalising $\phi$.
\end{lem}
\begin{proof}
    We first show that $\phi^*$ is injective.
    Let $(a_s)_{s:\sigma}$ and $(b_s)_{s:\sigma}$ be elements of $\int_{s:\sigma}A_s$ with the same image by $\phi^*$, that is, $(a_{\phi t})_{t:\tau} = (b_{\phi t})_{t:\tau}$. Then there is a $\tau$-large set of $t$ such that $a_{\phi t} = b_{\phi t}$; and since $\phi$ sends $\tau$-large sets to $\sigma$-large sets, we get $a_s = b_s$ for $\sigma$-all $s$.
    We now prove the characterisation of the image of $\phi^*$. An element in the image clearly satisfies the condition. Conversely, we suppose that $(a_t)\in \int_{t:\tau}A_{\phi t}$ cannot be lifted to $\int_{s:\sigma}A_s$. We choose underlying sets such that $\phi : T\to S$ is total, and we can check that the collection of subsets $T_0\times_{\phi(T_0)}T_0 \subseteq T \times_S T$, for $T_0$ $\tau$-large, together with the subset $D = \{(t,t')\in T \times_S T\ |\ a_t \neq a_{t'}\}$, satisfies \Cref{lem:extend-uf}. We can thus extend them into an ultrafilter $\lambda$ on $T\times_S T$. The sets $(T_0\times_{\phi(T_0)}T_0)$ ensure that the projections give maps $\psi_1,\psi_2 : \lambda\rightrightarrows \tau$ equalising $\phi$, and the set $D$ ensures that $\psi_1^*(a_t)\neq\psi_2^*(a_t)$.
\end{proof}
\begin{prop}\label{prop:sober-taut}
    Let $\E$ be a topos, its vu-category of points $\pt(\E)$ is taut.
\end{prop}
\begin{proof}
    We take a vu-arrow $f : p \vuto{\tau}q_{\phi t}$ in $\pt(\E)$, that is, maps $(f_E : p^*E\to\int_{t:\tau}{q_{\phi t}^*E})$ natural in $E\in\E$, such that $f$ is reducible along $\phi : \tau\to\sigma$ in $\pt(\E)$.
    We need to show that for any $\xi\in p^*E$, $f_E(\xi)$ is in the image of $\int_{s:\sigma}{q_s}^*E\hookrightarrow\int_{t:\tau}{q_{\phi t}}^*E$. 
    %
    The hypothesis that $f$ is reducible along $\phi$ ensures that $\psi_1^*f_E(\xi) = \psi_2^*f_E(\xi)$ for all $\psi_1,\psi_2 : \lambda\rightrightarrows \tau$ equalising $\phi$, and we conclude with \Cref{lemma:taut-sober}.
\end{proof}

\subsection{Reconstruction theorem}

\begin{defn}
    A \emph{vu-sheaf} over a vu-category $\X$ is a vu-functor $A : \X\to\Set$. We denote by $\Sh(\X)\coloneqq\vuCAT(\X,\Set)$ the category of vu-sheaves over $\X$.
\end{defn}

\begin{nota}
    For $A$ a vu-sheaf over $\X$ and $p$ a point of $\X$, the set $A(p)$ will also be denoted by $A_p$. For a vu-arrow $f : p\vuto{\sigma}q_s$ and $\xi\in A_p$, the $\sigma$-family $A(f)(\xi)\in\int_{\sigma}A_{q_s}$ will also be denoted by $f_*\xi$.
\end{nota}

\begin{ex}
    For $E\in\E$, we define the vu-sheaf $\ev(E)\in\Sh(\pt(\E))$ on the vu-category of points of $\E$ by $\ev(E)(p) \coloneqq p^*E$. This construction induces a functor of categories $\E\to\Sh(\pt(\E))$ (see \cite[5.2]{Saa}).
\end{ex}

We recall that the category $\Sh(\X)$ is always an infinity-pretopos (see \cite[5.7]{Saa} and \cite[4.3]{SUJ}), hence to be a topos only lacks its local presentability. A vu-category $\X$ is said \emph{bounded} if its category of vu-sheaves is locally presentable, and hence a Grothendieck topos. We will denote by $\vuCatb\subseteq\vuCAT$ the 2-full subcategory of bounded vu-categories, and we get a 2-functor $\Sh : \vuCatb\to\Topos$.

\begin{rmk}
    The term \emph{bounded} is used in \cite[5.8]{Saa} and is inspired by the analogue notion for Garner's ionad \cite[3.7]{IonadGarner}. We would however prefer an intrinsic notion of \emph{small generated vu-category} that can be directly seen on the vu-category (by opposition to boundedness that is seen on the category of vu-sheaves) and that would in particular imply boundedness. Such a notion does not exist yet, and as the present paper focus on the relation between vu-categories and toposes, the notion of boundedness will be sufficient.
\end{rmk}

The main result of \cite{Saa,Hamad2,SUJ} can then be stated as follows.
\begin{Thm}[Reconstruction theorem]\label{thm:reconstruction}
    The 2-functor $\pt : \Topos \to\vuCAT$ corestricts to bounded vu-categories and is right adjoint to $\Sh : \vuCatb\to\Topos$; moreover, this 2-adjunction is pseudo-idempotent and the fixed points on the $\Topos$ side are precisely the toposes with enough points.
\end{Thm}

In other words, we have the following 2-reflection:
\[\begin{tikzcd}[sep = 40pt]
    \Toposwep & \vuCatb.
    \arrow[""{name=0, anchor=center, inner sep=0}, "\pt"', shift right, hook, from=1-1, to=1-2]
    \arrow[""{name=1, anchor=center, inner sep=0}, "\Sh"'{pos=0.5}, shift right, curve={height=6pt}, from=1-2, to=1-1]
    \arrow["\dashv"{anchor=center, rotate=-91}, draw=none, from=1, to=0]
\end{tikzcd}\]
%
%
%
Hence, the counit $\ev : \E\to\Sh(\pt(\E))$ is an equivalence precisely if $\E$ has enough points, however the unit $\eta_\X : \X\to\pt(\Sh(\X))$ is not necessarily an equivalence. The next definition can be understood as a ‘dual’ of the toposic notion to have enough points.
%



\begin{defn}\label{def:enough-sheaves}
    A class of vu-sheaves over $\X$ is said \emph{to separate vu-arrows}, if for parallel vu-arrows $f,g$ in $\X$ to be equal, it is enough to check that $A_f = A_g$ on the vu-sheaves $A$ in the class.
    A vu-category is said to \emph{have enough sheaves} if the class of all vu-sheaves over it separates vu-arrows. 
\end{defn}

\begin{rmk}\label{rmk:enough-sheaves-unit}
    For a bounded vu-category $\X$ having enough sheaves is equivalent to $\eta_\X : \X\to\pt(\Sh(\X))$ being faithful. We will see in \Cref{locsobercaract} that, assuming $\X$ is taut, this is equivalent to the a priori stronger condition of $\eta_\X : \X\hookrightarrow\pt(\Sh(\X))$ being fully faithful.
\end{rmk}

\begin{ex}\label{counterex}
    Not every vu-category has enough sheaves; we give here a toy example of such a vu-category without enough sheaves. Take $K$ a non-trivial connected compact Hausdorff topological space, for example the unit interval. We consider $\X$ the vu-category with two objects $a$ and $b$ and $\Hom_\sigma(a,b) \coloneqq \pt(K)$ the underlying set of $K$ for any ultrafilter $\sigma$, and for $(x_s)_{s:\sigma}$ in $K$, the composite
    \[a\vuto[\delta]{\sigma}(a\vuto[x_s]{\tau_s}b) \text{ given by } a\vuto[y]{\sigma\cdot\tau_s}b\]
    where $y$ is the limit of $(x_s)_{s:\sigma}$ in $K$.
    One can check that this data is enough to define the vu-category $\X$ and that a vu-sheaf over $\X$ corresponds to the data of two sets $A$ and $B$ and a continuous map from $K$ into $B^A$ (with the product topology of $B$ seen as a discrete space). But as $K$ is connected and $B^A$ is totally disconnected, such a continuous map is always constant. Hence $\Sh(\X)$ is the Sierpiński topos $\Set^{[1]}$, where $[1]$ denotes the walking arrow category. Hence, the unit $\eta_\X : \X\to\pt(\Sh(\X)) = [1]$ identifies together all parallel vu-arrows of $\X$.
\end{ex}

\section{Topological reflection}\label{sec:vu-pos}

That toposes are a categorification of locales can be seen in terms of the 2-categories they form. The category of sheaves on a locale forms a topos ––– toposes arising in this way are called \emph{localic} --- and one can check that this assignment induces an embedding $\Loc \hookrightarrow \Topos$. The frame of opens of a locale can be recovered from its topos of sheaves as the frame of subterminal objects, and more generally, for any topos $\E$, the frame $\Sub_\E(1)$ is the universal locale associated to it.

\begin{nota}\label{prop:Lref-adjoint}
    We denote by $\Linj: \Loc \hookrightarrow \Topos$ the 2-fully faithful 2-functor sending a locale to its topos of sheaves. The assignment $\E \mapsto \Sub_\E(1)$ induces a 2-functor $\Lref{-}: \Topos \to \Loc$ which is left adjoint to $\Linj$. In other words, we get the following 2-reflection:
    \[\begin{tikzcd}[sep = 40pt]
    	\Loc & \Topos.
    	\arrow[""{name=0, anchor=center, inner sep=0}, "\Linj"', shift right, hook, from=1-1, to=1-2]
    	\arrow[""{name=1, anchor=center, inner sep=0}, "\Lref{-}"'{pos=0.45}, shift right, curve={height=6pt}, from=1-2, to=1-1]
    	\arrow["\dashv"{anchor=center, rotate=-91}, draw=none, from=1, to=0]
    \end{tikzcd}\]
\end{nota}
Our objective is to show that a similar relation holds between topological spaces and vu-categories: the former form a reflective 2-subcategory of the latter (see \Cref{prop:Tref-adjoint}).

\subsection{Vu-posets}

\begin{defn}
    A vu-category $\X$ is \emph{posetal} if each $\Hom_\sigma(p, q_s)$ is a subsingleton, we write $p \vuleq{\sigma} q_s$ if this set is inhabited. A \emph{vu-poset} is a posetal vu-category, the full subcategory of $\vuCAT$ of vu-posets (resp. essentially small vu-posets) is denoted by $\vuPOS$ (resp. $\vuPos$).
\end{defn}

We will use $X,Y,\dots$ rather than  $\X, \Y,\dots$ to denote vu-posets. The intuition behind $p \vuleq{\sigma} q_s$ is that $(q_s)_{s:\sigma}$ is an ultrasequence of points which converges to $p$.

\begin{defn}
    Let $X$ be a topological space. Write $p \vuleq[X]{\sigma} q_s$, and we say that \emph{$(q_s)_{s:\sigma}$ converges to $p$}, when $p \in U$ implies $(\forall s:\sigma)\ q_s \in U$ for each open $U$ of $X$.
\end{defn}

The resulting structure defines a vu-poset on the points on $X$. The vu-poset obtained is the vu-category of points of the topos of sheaves above $X$ restricted to the points of $X$. 

\begin{rmk}\label{reconstruction-thm-topo}
    The reconstruction theorem (\Cref{thm:reconstruction}), restricted to this context, identifies sheaves over a topological space and vu-sheaves over the associated vu-poset. Two different direct proofs of this result can be found in \cite[5.4]{Saa} and in \cite[5.7]{Hamad2}.
\end{rmk}

In particular the opens of a topological space $X$ are completely determined by the relation $\vuleq[X]{}$, in the following sense (see for example \cite[2.5]{Saa}).

\begin{prop}\label{prop:opens-are-upsets}
    Let $X$ be a topological space, then a subset $A$ of $X$ is open if and only if $p \in A$ and $p \vuleq[X]{\sigma} q_s$ implies $(\forall s:\sigma)\ q_s \in A$.
\end{prop}

We can deduce that the 2-functor sending a topological space $X$ to the vu-poset $(X,\vuleq[X]{})$ is 2-fully faithful.

\begin{prop}
    Let $X$ be a topological space, $p$ and $(q_s)_{s:\sigma}$ points of $X$, and $\phi : \tau \to \sigma$. Then $p \vuleq[X]{\tau} q_{\phi t}$ implies $p \vuleq[X]{\sigma} q_s$.
\end{prop}
\begin{proof}
    It is enough to see that for any subset $U$ of $X$, $(\forall t:\tau)\ q_{\phi t}\in U$ if and only if $(\forall s:\sigma)\ q_{s}\in U$. And both these conditions are equivalent to $U$ being large for the ultrafilter $q_*(\phi_*(\tau)) = q_*(\sigma)$.
\end{proof}
Hence, in the vu-poset of points of a topological space, any vu-arrow can be unthickened as much as possible. Actually, tautness simplifies in the posetal case, and the above proposition expresses that the vu-poset of points of a topological space is taut.
\begin{prop}\label{def:taut-poset}
    A vu-poset $(X,\vuleq{})$ is taut if and only if its vu-inequalities can always be unthickened as much as possible, \ie if $p \vuleq{\tau} q_{\phi t}$ implies $p \vuleq{\sigma} q_s$ for $\phi : \tau \to \sigma$. 
\end{prop}
\begin{proof}
    It is enough to notice that in a vu-poset, a vu-arrow $f: a\vuto{\tau}b_{\phi t}$ is always reducible along $\phi : \tau\to\sigma$.
\end{proof}




%


\subsection{Topological concepts on vu-posets}

In light of \Cref{prop:opens-are-upsets}, we can define the notion of open subset for an arbitrary vu-poset --- a similar definition has been given in \cite[3.23]{SUJ}.  

\begin{defn}\label{def:open}
    An \emph{open}, also called an \emph{upset}, $U$ of a vu-poset $X$ is a subposet of $X$ such that if $p \in U$ and $p \vuleq{\sigma} q_s$ then $q_s \in U$ for all $s : \sigma$. We denote by $\Open(X)$ the poset of opens of $X$.
    Similarly, a subposet $F$ of $X$ is \emph{closed} if it is downward-closed, that is, if $p \vuleq{\sigma} q_s$ and $(\forall s:\sigma)\ q_s\in F$ imply $p \in F$.
    And the \emph{closure} of a collection $A$ of points of a vu-poset $X$, denoted by $\cl(A)$, is the collection of $p\in X$ such that there exists an ultrafamily $(a_s)_{s:\sigma}$ of points in $A$ such that $p \vuleq{\sigma} a_s$.
\end{defn}
\begin{rmk}
    The collection of opens of a vu-poset is not necessarily small, for instance those of a large discrete vu-poset. We will have more to say about vu-posets which possess only a set's worth of opens later.
\end{rmk}

\begin{lem}\label{lem:smallest-closedset}
    Let $A$ be a collection of points of a vu-poset $X$, $\cl(A)$ is the smallest closed subposet of $X$ containing $A$.
\end{lem}

\begin{lem}\label{lem:closure}
    Let $A$ be a collection of points of a vu-poset $X$, a point $p\in X$ belongs to $\cl(A)$ if and only if every open $U \subseteq X$ which contains $p$ intersects $A$. 
\end{lem}
\begin{proof}
    If $p\in\cl(A)$, then $p\vuleq{\sigma}a_s$ for some $(a_s)_{s:\sigma}$ in $A$, and so any open $U$ containing $p$ contains the $(a_s)$ and so intersects $A$. Conversely if $p\notin\cl(A)$, then $X\setminus\cl(A)$ is an open containing $p$ that does not intersect $A$.
\end{proof}

\Cref{prop:opens-are-upsets} shows that if one starts with a topological space $X$, sees it as a vu-poset with $\vuleq[X]{}$, and then back as a topological space whose opens are the upsets, then one gets back to $X$ itself. We now turn to the analogue question for vu-posets, is it true that $\vuleq{}$ can be fully recovered by the opens it induces? The next proposition answers this question in the positive for taut vu-posets.

 

\begin{prop}\label{prop:Barr}
    Let $X$ be a vu-poset. If $p$ and $(q_s)_{s \in \sigma}$ in $X$ are such that every open $U$ of $X$ which contains $p$ also contains $(q_s)_{s\in\sigma}$, then $p \vuleq{\tau} q_{\phi t}$ for some thickening $\phi : \tau \to \sigma$.
\end{prop}

\begin{proof}
    We fix $S$ an underlying set of $\sigma$.
    
    First, we notice that if $\ell \subseteq S$ is $\sigma$-large then $p \in \cl(\{q_s : s \in \ell\})$. Indeed, if an upset $U$ contains $p$ then $\{s \in \ell : q_s \in U\}$ is $\sigma$-large as the intersection of two large sets, hence non-empty --- thus  \Cref{lem:closure} ensures $p \in \cl(\{q_s : s \in \ell\})$. Hence we get some $(s^\ell_t \in \ell)_{t:\tau_{\ell}}$ such that \[p \vuleq{\tau_{\ell}} q_{s^\ell_t}.\]

    Using \ref{lem:ext-final-uf}, we can consider an initial ultrafilter $\lambda$ on the codirected poset of $\sigma$-large sets; we can thus merge the above $(s^\ell_t)\in S^{\tau_\ell}$ together along $\lambda$ into the family
    \[\phi \coloneqq (s^\ell_t)_{(\ell,t) : \tau} \in S^\tau\]
    where $\tau\coloneqq\lambda\cdot\tau_\ell$.
    
    We show it is a thickening of $\sigma$, \ie that $\phi_*(\tau) = \sigma$.
    It is enough to show, for each $\sigma$-large $\ell_0\subseteq S$, that $(\forall (\ell,t) : \tau)\ \phi{(\ell,t)}\in \ell_0$, \ie that
    \[(\forall \ell : \lambda)\ (\forall t : \tau_\ell)\ s^\ell_t\in \ell_0,\]
    and for $\ell: \lambda$, we can assume $\ell\subseteq \ell_0$ by initiality of $\lambda$, and then for $t:\tau_\ell$ we have $s^{\ell}_t \in \ell \subseteq \ell_0$ as desired.
    
    Hence, $\phi$ induces a map $\phi:\tau\to\sigma$, and the composite
    \[p \vuleq{\ell:\lambda} (p \vuleq{t:\tau_{\ell}} q_{s^{\ell}_t})\]
    gives $p \vuleq{\ell:\lambda\cdot t:\tau_\ell} q_{s^{\ell}_t}$, \ie $p \vuleq{(\ell,t):\tau} q_{\phi(\ell,t)}$ as wanted.
\end{proof}

\begin{rmk}
    We should stress the importance of this proposition as it enables us to get a vu-inequality (and thus a vu-arrow) from an assumption on the opens. This result is already essentially in \cite{Barr}, and will be categorified later in \Cref{prop:Barr-categorified} (see also \Cref{rmk:Barr-categorified}).
\end{rmk}

For a taut vu-poset we deduce the following.
\begin{cor}\label{cor:Barr-taut} \label{cor:poset-iso-points}
    For $p$ and $(q_s)_{s:\sigma}$ points of a taut vu-poset, $p\vuleq{\sigma}q_s$ if and only if every open containing $p$ also contains $(q_s)_{s:\sigma}$. In particular, two points of a taut vu-poset which belong to the same opens are isomorphic.
\end{cor}
\begin{proof}
    Suppose that every open containing $p$ also contains $(q_s)_{s:\sigma}$, then \Cref{prop:Barr} gives us $p \vuleq{\sigma} q$ and $q \vuleq{\tau} p$, and by tautness of the vu-poset (\Cref{def:taut-poset}) we deduce $p \vuleq{\ptuf} q$ and $q \vuleq{\ptuf} p$. 
\end{proof}

We are now able to prove the following proposition, which 
concludes the characterisation of the essential image of $\TopSp$.

\begin{cor}\label{prop:small-vu-pos}
    For $X$ a taut vu-poset, the following conditions are equivalent,
    \begin{enumerate}      
        \item $X$ has a set's worth of opens; \label{cond:poset-small-opens}
        \item $X$ is essentially small; \label{cond:poset-small}
        \item $X$ belongs to the essential image of $\TopSp$. \label{cond:poset-TopSp}
    \end{enumerate}
\end{cor}
\begin{proof}
    By \Cref{cor:poset-iso-points}, the isomorphism class of a point of $X$ is completely determined by the way it lies in the opens of $X$. If \pref{cond:poset-small-opens} holds, then there is at most a set's worth of ways to lie in the opens, and thus the isomorphism classes of $X$ form a set, \ie \pref{cond:poset-small} holds.
    If \pref{cond:poset-small} holds, we can suppose up to equivalence that $X$ is a small taut vu-poset. Then the upsets on $X$ define a topology on the set of points of $X$ and \Cref{cor:Barr-taut} shows that $X$ is the vu-poset of points of this topological space, \ie \pref{cond:poset-TopSp} holds.
    Finally, {\pref{cond:poset-TopSp} $\Rightarrow$ \pref{cond:poset-small-opens}} follows from \Cref{prop:opens-are-upsets}.
\end{proof}

In the rest of this paper, in view of the previous proposition, we blur the distinction between topological spaces and essentially small taut vu-posets. That is, we identify the two 2-category $\TopSp$ with the 2-category $\vuPostaut$ of essentially small taut vu-posets.

\subsection{The topological reflection}

We can now define the topological reflection of a vu-category, the vu-categorical counterpart of the localic reflection of a topos.

\begin{defn}\label{def:Tref}
    Let $\X$ be a vu-category, its \emph{topological reflection}, denoted by $\Tref{\X}$, is the vu-poset on the objects of $\X$ where $p \vuleq{\sigma}q_s$ if there exists a vu-arrow $p \vuto{\tau} q_{\phi t}$ in $\X$ for some thickening $\phi : \tau \to \sigma$.
\end{defn}

\begin{prop}
    The topological reflection $\Tref{\X}$ is taut and the natural vu-functor $\X \to \Tref{\X}$ is universal among vu-functors from $\X$ into a taut vu-poset, that is, for a taut vu-poset $X$, any vu-functor $\X\to X$ factorises uniquely trough $\X \to \Tref{\X}$.  
\end{prop}
\begin{proof}
    Suppose $p\vuleq{\tau}q_{\phi t}$ in $\Tref{\X}$ for some thickening $\phi : \tau\to\sigma$. This means that there exists a vu-arrow $p\vuto{\lambda}q_{\psi \phi \ell}$ for some $\psi : \lambda\to\tau$, and thus $p\vuleq{\sigma}q_s$ in $\Tref{\X}$. Hence, by \Cref{def:taut-poset}, $\Tref{\X}$ is taut.

    We now show the universal property. Let $X$ be a taut vu-poset; we show that any vu-functor $\X\to X$ factorises uniquely through $\X \to \Tref{\X}$. The uniqueness follows from the fact that a functor into a vu-poset is determined by its behaviour on the objects and that $\X \to \Tref{\X}$ is surjective on objects.
    We take a vu-functor $F : \X\to X$ and we show that $F$ induces a map of vu-posets $\Tref{\X}\to X$, \ie that $p\vuleq[\Tref{\X}]{\sigma}q_s$ implies $F(p)\vuleq[X]{\sigma}F(q_{s})$.
    By definition, $p\vuleq[\Tref{\X}]{\sigma} q_s$ gives some map $\phi:\tau\to\sigma$ and some $p\vuto{\tau}q_{\phi t}$ in $\X$, hence $F(p)\vuleq[X]{\tau}F(q_{\phi t})$, and by \Cref{def:taut-poset} we have $F(p)\vuleq[X]{\sigma}F(q_{s})$.
\end{proof}

\begin{prop}\label{circular-square}
    Let $\X$ be a bounded vu-category, $\shloc(\Tref{\X}) \cong \Lref{\shvu(\X)}$.
\end{prop}
\begin{proof}
    We denote by $[1]$ the taut vu-poset corresponding to the Sierpiński space, so that
    \[\Open(\Tref{\X})\cong\vuCat(\Tref{\X},[1])\cong\vuCat(\X,[1])\cong\Sub_{\Sh(\X)}(1) \cong \Lref{\shvu(\X)}.\qedhere\]
\end{proof}

\begin{prop}\label{prop:Tref-adjoint}
    If moreover $\X$ is bounded, we can identify $\Tref{\X}$ with a topological space, and the assignment $\X \mapsto \Tref{\X}$ extends to a 2-functor $\Tref{-} : \vuCatb \to \TopSp$ that is left adjoint to the inclusion $\Tinj : \TopSp \hookrightarrow \vuCatb$, \ie we have a 2-reflection:
    \[\begin{tikzcd}
    	\TopSp & \vuCatb.
    	\arrow[""{name=0, anchor=center, inner sep=0}, "\Tinj"',shift right, hook, from=1-1, to=1-2]
    	\arrow[""{name=1, anchor=center, inner sep=0}, "\Tref{-}"', shift right, curve={height=6pt}, from=1-2, to=1-1]
    	\arrow["\dashv"{anchor=center, rotate=-90}, draw=none, from=1, to=0]
    \end{tikzcd}\]
\end{prop}
\begin{proof}
    We suppose $\X$ bounded. By \Cref{circular-square} $\Tref{\X}$ has a set's worth of opens, and by \Cref{prop:small-vu-pos} it can be identified to a topological space.
    The 2-functoriality and the reflection are straightforward.
\end{proof}

\begin{rmk}
    Although we defined $\Tref{\X}$ for all vu-categories, we only consider it as a functor on $\vuCatb$ and not on $\vuCAT$.
    This restriction to bounded vu-categories is only here to ensure that the reflection functor lands in $\TopSp$ and not in some category of ``large'' topological spaces.
    That being said, it is worth noting that the adjunction of \Cref{prop:Tref-adjoint} is purely algebraic: $\Tref{-} : \vuCAT \to \vuPOStaut$ is left adjoint to the inclusion $\vuPOStaut \injto \vuCAT$, and that, if $X$ is a vu-poset, the reflection $X\to\Tref{X}$ can be seen as the \emph{tautification} of the vu-poset $X$.
\end{rmk}

We can picture all the functors on stage in the following square of adjunctions.
\[\begin{tikzcd}[sep = 50pt]
    \Loc & \TopSp \\
    \Topos & \vuCatb
    \arrow[""{name=0, anchor=center, inner sep=0}, "{\ptloc}"', from=1-1, to=1-2]
    \arrow[""{name=1, anchor=center, inner sep=0}, "\Linj", hook, from=1-1, to=2-1]
    \arrow[""{name=2, anchor=center, inner sep=0}, "{\shloc}"', curve={height=18pt}, dashed, from=1-2, to=1-1]
    \arrow[""{name=3, anchor=center, inner sep=0}, "\Tinj"', hook, from=1-2, to=2-2]
    \arrow[""{name=4, anchor=center, inner sep=0}, "\Lref{-}", curve={height=-18pt}, dashed, from=2-1, to=1-1]
    \arrow[""{name=5, anchor=center, inner sep=0}, "{\ptvu}", from=2-1, to=2-2]
    \arrow[""{name=6, anchor=center, inner sep=0}, "\Tref{-}"', curve={height=18pt}, dashed, from=2-2, to=1-2]
    \arrow[""{name=7, anchor=center, inner sep=0}, "{\shvu}", curve={height=-18pt}, dashed, from=2-2, to=2-1]
    \arrow["\vdash"{anchor=center, rotate=90}, draw=none, from=0, to=2]
    \arrow["\dashv"{anchor=center, rotate=0}, draw=none, from=1, to=4]
    \arrow["\vdash"{anchor=center, rotate=0}, draw=none, from=3, to=6]
    \arrow["\dashv"{anchor=center, rotate=90}, draw=none, from=5, to=7]
\end{tikzcd}\]

The solid arrow square $\ptvu(\Linj(L))\simeq\Tinj(\ptloc(L))$ and its adjoint $\Lref{\shvu(\X)} \simeq \shloc(\Tref{\X})$ commute up to some natural equivalence.
There are also two other squares mixing left and right adjoints, corresponding to the Beck--Chevalley conditions.
\[\begin{tikzcd}[sep = 40pt]
	\Loc & \TopSp \\
	\Topos & \vuCatb
	\arrow["\Linj", hook, from=1-1, to=2-1]
	\arrow["{\shloc}"', dashed, from=1-2, to=1-1]
	\arrow["\Tinj"', hook, from=1-2, to=2-2]
	\arrow["{\shvu}", dashed, from=2-2, to=2-1]
	\arrow[between={0.4}{0.6}, Rightarrow, from=2-2, to=1-1]
\end{tikzcd}
\hspace{40pt}
\begin{tikzcd}[sep = 40pt]
	\Loc & \TopSp \\
	\Topos & \vuCatb
	\arrow["{\ptloc}"', from=1-1, to=1-2]
	\arrow["\Lref{-}", dashed, from=2-1, to=1-1]
	\arrow["{\ptvu}", from=2-1, to=2-2]
	\arrow["\Tref{-}"', dashed, from=2-2, to=1-2]
	\arrow[between={0.4}{0.6}, Rightarrow, from=2-2, to=1-1]
\end{tikzcd}\]

The first map $\shvu(\Tinj(T))\to\Linj(\shloc(T))$ is given by its inverse image, sending a sheaf over $T$ to the vu-functor $\Tinj(T)\to\Set$ given on objects by the germs of the sheaf, and \Cref{reconstruction-thm-topo} says precisely that this map is an equivalence of categories.

The second map $\Tref{\ptvu(\E)} \to \ptloc(\Lref{\E})$ maps points of $\E$ to their image in $\Lref{\E} $. It has been shown in \cite[Lem. 5.3]{SUJ} that it is always fully faithful. 
Concretely, if $p \vuleq{\sigma} q_s$ in $\pt(\Lref{\E} )$ then there exists a map $p \vuto{\tau} q_{\phi t}$ in $\vupt(\E)$ for some $\phi : \tau \to \sigma$.

\begin{prop}\label{prop:rpt=ptr}
    The vu-functor between vu-posets $\Tref{\ptvu(\E)} \to \ptloc(\Lref{\E})$ is fully faithful.
\end{prop}

\begin{rmk}
    Note however that $\Tref{\ptvu(\E)} \to \ptloc(\Lref{\E})$ is not necessarily an equivalence: a point of $\Lref{\E}$ does not necessarily come from a point of $\E$. For example, $\E$ may be a hyperconnected topos (so $\pt(\Lref{\E}) = 1$) with no points (so $\Tref{\ptvu(\E)}$ is the empty space). The existence of such a topos follows from \cite[3.4.14 and 3.5.4]{johnstone:elephant}.
\end{rmk}

\section{Easy preservation results}\label{sec:ff}

The goal of the paper is to give vu-categorical analogues of the usual classes of geometric morphisms between toposes. In this section we begin with some easy preservation results. First we show how étale geometric morphisms and étale vu-functors relate (\Cref{thm:etale}). Then we show that if a geometric morphism is an embedding (resp. localic) then it is fully faithful (resp. faithful) on its points (\Cref{cor:emb-ff} and \Cref{thm-loc}).

\subsection{Étale vu-functors}\label{subsec:etale}

We recall that étale geometric morphisms are geometric morphisms of the form $f : \E/X\to\E$ for some $X\in\E$.  
%
%
A vu-counterpart of étale geometric morphisms has already been introduced in \cite[Section 4]{SUJ}. 

\begin{defn}
    A vu-functor $F : \A\to\X$ is said \emph{étale} if all its fibres $F^{-1}\{p\}\coloneqq \{a\in\A \ | \ F(a) = p\}$ are small, and if any $g : F(a)\vuto{\sigma}q_s$ has a unique lift, \ie there exists a unique vu-arrow $f : a\vuto{\sigma}b_s$ in $\A$ such that $F(f) = g$. We will write $g_*(a)$ or $g_*a$ for the family $(b_s)_{s:\sigma}$.
\end{defn}

We also recall from \cite[Section 4]{SUJ} that vu-sheaves over $\X$ can be seen via their total space.
From a vu-sheaf $A : \X\to\Set$ we can construct an étale vu-functor $\A\to\X$ with $\A$ the \emph{total space} of the vu-sheaf $A$.
The objects of $\A$ are pairs $(p\in \X, \xi\in A(p))$, $\sigma$-arrows from $(p,\xi)$ to $(q_s,\zeta_s)_{s:\sigma}$ are $f\in\Hom_\sigma(p,q_s)$ such that $f_*\xi = (\zeta_s)_{s:\sigma}$, and the vu-functor $\A\to\X$ is given by the obvious projection.
This construction induces an equivalence from the category of vu-sheaves $\Sh(\X)$ to the category of étale vu-functors above $\X$; we refer to \cite[Theorem 4.15]{SUJ} for a detailed proof.
This construction can be seen as the vu-categorical analogue of the Grothendieck construction relating copresheaves and discrete opfibrations over a category, but also as the categorification of the correspondence between sheaves over a topological space and local homeomorphisms above it.

\begin{nota}
    The total space of a vu-sheaf $A,B,\dots$ above $\X$ will be systematically denoted by the corresponding letter $\A,\B,\dots$, or by $\et{A}, \et{B}, \dots$ if we need to be explicit.
\end{nota}

\begin{lem}\label{lem:slice-etale}
    Let $F : \B\to\A$ and $G:\A\to\X$ be two vu-functors,
    \begin{enumerate}
        \item if $F$ and $G$ are étale, so is $G\circ F$,
        \item if $G\circ F$ and $G$ are étale, so is $F$.
    \end{enumerate}
\end{lem}
\begin{proof}\leavevmode
    \begin{enumerate}
        \item The fibre $(G\circ F)^{-1}\{p\} = \bigsqcup_{a\in G^{-1}\{p\}} F^{-1}\{a\}$ is small as a small union of small sets, and the unique lifting properties compose well.
        \item The fibre $F^{-1}\{a\}$ is small as a subset of $(G\circ F)^{-1}\{F(a)\}$, and the unique lifting of some $f : F(x)\vuto{\sigma} b_s$ is given by the unique lifting of $G(f) : G\circ F(x)\vuto{\sigma}F(b_s)$ along $G\circ F$.\qedhere
    \end{enumerate}
\end{proof}

We present now our very first result relating geometric morphisms and vu-functors.

\begin{Thm}\label{thm:etale}\leavevmode
    \begin{enumerate}
        \item\label{prop:slice-etale} Let $A\in\Sh(\X)$ be a vu-sheaf corresponding to $F : \A\to\X$ étale. Then $\Sh(F) : \Sh(\A)\to\Sh(\X)$ can be identified with the étale geometric morphism $\Sh(\X)/A\to\Sh(\X)$.
        
        \item\label{prop:point-etale} Let $\E/X\to\E$ be an étale geometric morphism, then $\pt(\E/X)\to\pt(\E)$ is the étale vu-functor corresponding to the vu-sheaf $\ev(X)\in\Sh(\pt(\E))$.
    \end{enumerate}
\end{Thm}
\begin{proof}\leavevmode
    \begin{enumerate}
        \item 
        The functor $\Sh(\A)\to\Sh(\X)/A$ sends an étale map $\Y\to\A$ to the composite $\Y\to\A\to\X$ that is étale by \Cref{lem:slice-etale}. 
        Conversely, from an object in $\Sh(\X)/A$ we get an étale vu-functor $\B\to\X$ together with a vu-functor $\B\to\A$ above $\X$, and by \Cref{lem:slice-etale} $\B\to\A$ is étale, thus giving a vu-sheaf over $\A$.
        
        \item We only describe $\pt(\E/X)$, the rest follows easily. The objects of $\pt(\E/\X)$ are pairs $(p,\xi)$ of a point $p$ of $\E$ with a germ $\xi\in p^*X$, and $\sigma$-arrows from $(p,\xi)$ to $(q_s,\zeta_s)$ are given by $f\in\Hom^{\pt(\E)}_\sigma(p,q_s)$, that is natural transformations $f : p^*\Rightarrow\int_{s:\sigma}({q_s}^*)$, such that $f_X(\xi) = (\zeta_s)_{s:\sigma}$.\qedhere
    \end{enumerate}
\end{proof}

We can now fill the first line in the table (\ref{tableau}), summarising the relation between étale vu-functors and étale geometric morphisms.
\begin{cor}\label{table:etal}
    \ 
    
    \centerline{
    \begin{tabular}{r l||C C|C C|C C|}
    & & \multicolumn{2}{c|}{$f\squig\pt(f)$}&
        \multicolumn{2}{c|}{$\Sh(F)\squig F$}&
        \multicolumn{2}{c|}{$f\squig F$}\\
    \cline{3-8}
    geometric morphisms & vu-functors & $\Rightarrow$ & $\Leftarrow$ & $\Rightarrow$ & $\Leftarrow$ & $\Rightarrow$ & $\Leftarrow$ \\
    \cline{1-8}
    étale & étale
        & \checkmark & \checkmark 
        & $\crossmark$ & \checkmark
        & $\crossmark$ & \checkmark
    \end{tabular}}
\end{cor}
\begin{proof}
    All \checkmark can be deduced from the reconstruction theorem (\Cref{thm:reconstruction}) and \Cref{thm:etale}. 
    For both counterexamples, take $\X$ from \Cref{counterex}, then the unit $F : \X\to\ [1]$ is not étale, although $f : \Sh(\X)\to\Set^{[1]}$ is an equivalence of toposes and so in particular is étale.
\end{proof}

\subsection{Embeddings}

We recall that a geometric morphism $j : \E \to \F$ is an \emph{embedding} if the right adjoint $j_* : \E\to\F$ is fully faithful. We can abstract the categorical result needed to show that $\pt(j)$ is then fully faithful in the following lemma.
\begin{lem}\label{prop:fully faithful general}
    Let $L : \D \to \C$ be a functor between categories having a fully faithful right adjoint, and let $F,G : \C \to \A$ be two functors to another category $\A$. For any natural transformation $\beta : FL \Rightarrow GL$ there exists a unique natural transformation $\alpha : F \Rightarrow G$ such that $\alpha L = \beta$. 
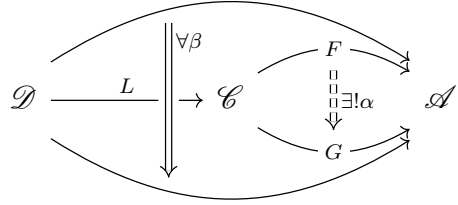
\[\begin{tikzcd}
	\D && \C && \A
	\arrow[""{name=0, anchor=center, inner sep=0}, "L", from=1-1, to=1-3]
	\arrow[""{name=1, anchor=center, inner sep=0}, shift left=3, curve={height=-30pt}, from=1-1, to=1-5]
	\arrow[""{name=2, anchor=center, inner sep=0}, shift right=3, curve={height=30pt}, from=1-1, to=1-5]
	\arrow[""{name=4, anchor=center, inner sep=0}, "F"{description}, shift left=2, curve={height=-12pt}, from=1-3, to=1-5]
	\arrow[""{name=5, anchor=center, inner sep=0}, "G"{description}, shift right=2, curve={height=12pt}, from=1-3, to=1-5]
	\arrow["{\forall \beta}"{pos=0.2}, shift right=8, between={0.1}{0.9}, crossing over, Rightarrow, from=1, to=2]
	\arrow["{\exists!\alpha}", between={0.2}{0.8}, Rightarrow, dashed, from=4, to=5]
\end{tikzcd}\]
\end{lem}
\begin{proof}
    We denote $R:\C\hookrightarrow\D$ the right adjoint of $L$.
    
    Suppose such an $\alpha$ exists. Since $R$ is fully faithful, the counit $\varepsilon_C : L R(C) \to C$ is an isomorphism. The naturality square it induces forces $\alpha$ to be $(G\varepsilon) \circ (\beta R) \circ (F\varepsilon)^{-1}$, thus $\alpha$ is unique.
    \[\begin{tikzcd}
    	{F(L R(C))} & {F(C)} \\
    	{G(L R(C))} & {G(C)}
    	\arrow["{F\varepsilon_C}", from=1-1, to=1-2]
    	\arrow[""{name=0, anchor=center, inner sep=0}, "{\alpha_{L R(C)}}", shift left=2, from=1-1, to=2-1]
    	\arrow[""{name=1, anchor=center, inner sep=0}, "{\beta_{R(C)}}"', shift right=2, from=1-1, to=2-1]
    	\arrow["{\alpha_C}", from=1-2, to=2-2]
    	\arrow["{G\varepsilon_C}"', from=2-1, to=2-2]
    	\arrow[between={0.3}{0.7}, equals, from=0, to=1]
    \end{tikzcd}\]
    
    Of course, the above square can just as well be taken as a definition of $\alpha$, which would settle existence provided that this definition satisfies $\alpha_{L D} = \beta_D$ for each $D \in \D$. For this, one needs only to paste the defining square for $\alpha_{L D}$ with the naturality square of $\beta$ at the unit $\eta_D$, noting that the long edges are identity morphisms due to the triangle identities of $L \dashv R$.
    \[\begin{tikzcd}
	{FL(D)} & {FLRL(D)} & {FL(D)} \\
	{GL(D)} & {GLRL(D)} & {GL(D)}
	\arrow["{FL\eta_D}", from=1-1, to=1-2]
	\arrow["{=}"{description}, curve={height=-22pt}, dotted, no head, from=1-1, to=1-3]
	\arrow["{\beta_D}"', from=1-1, to=2-1]
	\arrow["{F\varepsilon_{L D}}", from=1-2, to=1-3]
	\arrow["{\beta_{RL D}}"{description}, from=1-2, to=2-2]
	\arrow["{\alpha_{L D}}", from=1-3, to=2-3]
	\arrow["{GL\eta_D}"', from=2-1, to=2-2]
	\arrow["{=}"{description}, curve={height=22pt}, dotted, no head, from=2-1, to=2-3]
	\arrow["{G\varepsilon_{L D}}"', from=2-2, to=2-3]
    \end{tikzcd}\qedhere\]
\end{proof}

\begin{Thm} \label{cor:emb-ff}
    If $j : \E \to \F$ is an embedding of toposes, then $\pt(j) : \pt(\E) \to \pt(\F)$ is fully faithful. 
\end{Thm}
\begin{proof}
    Let $p$ and $(q_s)_{s:\sigma}$ be points of $\E$, we want to show that
    \[\Hom^{\pt(\E)}_\sigma(p, q_s) \to \Hom^{\pt(\F)}_\sigma(jp, jq_s)\]
    is a bijection. This follows from applying \Cref{prop:fully faithful general} to the adjunction $j^* \dashv j_* : \E \leftrightarrows \F$ with $F$ being $p^* : \E \to \Set$ and $G$ being $\int_\sigma({q_s}^*) : \E \to \Set$.
\end{proof}

We can summarise the relation between fully faithful vu-functors and embeddings of toposes as follows, thus filling a line in the table (\ref{tableau}).

\begin{cor}\label{table:emb}
    \
    
    \centerline{
    \begin{tabular}{r l||C C|C C|C C|}
    & & \multicolumn{2}{c|}{$f\squig\pt(f)$}&
        \multicolumn{2}{c|}{$\Sh(F)\squig F$}&
        \multicolumn{2}{c|}{$f\squig F$}\\
    \cline{3-8}
    geometric morphisms & vu-functors & $\Rightarrow$ & $\Leftarrow$ & $\Rightarrow$ & $\Leftarrow$ & $\Rightarrow$ & $\Leftarrow$ \\
    \cline{1-8}
    embeddings & fully faithful
        & \checkmark & \checkmark 
        & $\crossmark$ & ?
        & $\crossmark$ & \checkmark 
    \end{tabular}}

\end{cor}
\begin{proof}
    Both \checkmark can be deduced from the reconstruction theorem (\Cref{thm:reconstruction}) and \Cref{cor:emb-ff}. 
    For both counterexamples, take $\X$ from \Cref{counterex}, then the unit $F : \X\to\ [1]$ is not fully faithful, although $f : \Sh(\X)\to\Set^{[1]}$ is an equivalence of toposes and so in particular is an embedding.
\end{proof}

        %
        %

    %

\subsection{Localic geometric morphisms}

Our next objective is to show that localic geometric morphisms induce faithful vu-functors on their points. We recall that a geometric morphism $f : \E \to \F$ is \emph{localic} if every object of $\E$ is a subquotient of an object of the form $f^*(X)$ with $X \in F$.

\begin{defn}
    In a category, an object $A$ is a \emph{subquotient} of an object $B$ if it is a quotient of a subobject of $B$, \ie if there exists a span $B \hookleftarrow C \epi A$ whose left leg is a monomorphism and whose right leg is an epimorphism.
\end{defn}

\begin{rmk}\label{rem_subquotient-order}
    In a category where monomorphisms and epimorphisms are stable under pullbacks and pushouts, such as a topos, the order does not matter: every quotient of a subobject of $B$ is a subobject of a quotient of $B$ and vice versa.
\end{rmk}

Here is the categorical result needed.
\begin{lem}\label{lem_localic}
    Suppose that $\C$ is a category and that $\X$ is a collection of objects of $\C$ such that any object of $\C$ is a subquotient of an element of $\X$. Let $F,G : \C \to \D$ be two functors such that $F$ preserves epimorphisms and $G$ preserves monomorphisms and take $\alpha, \beta : F \Rightarrow G$ two natural transformations between them.
    If $\alpha_X = \beta_X$ for every $X \in \X$, then $\alpha = \beta$.
\end{lem}
\begin{proof}
    Pick $C \in \C$. By assumption there exist a monomorphism $m : A \mono X$ with $X \in \X$ and an epimorphism $e : A \epi C$. Consider the naturality squares of $\alpha$ and $\beta$ for $m : A \mono X$.
    \[\begin{tikzcd}
    	{F(A)} && {F(X)} \\
    	\\
    	{G(A)} && {G(X)}
    	\arrow["{F(m)}", from=1-1, to=1-3]
    	\arrow["{\alpha_A}"', shift right=2, from=1-1, to=3-1]
    	\arrow["{\beta_A}", shift left=2, from=1-1, to=3-1]
    	\arrow[""{name=0, anchor=center, inner sep=0}, "{\beta_X}", shift left=2, from=1-3, to=3-3]
    	\arrow[""{name=1, anchor=center, inner sep=0}, "{\alpha_X}"', shift right=2, from=1-3, to=3-3]
    	\arrow["{G(m)}", hook, from=3-1, to=3-3]
    	\arrow[between={0.2}{0.8}, equals, from=1, to=0]
    \end{tikzcd}\]
    From $\alpha_X = \beta_X$ it follows that $G(m) \circ \alpha_A = G(m) \circ \beta_A$, but $G(m)$ is a monomorphism by assumption, so $\alpha_A = \beta_A$.
    By the same reasoning on the naturality square of $e : A \epi C$, together with the fact that $F$ preserves epimorphisms, it follows that $\alpha_C = \beta_C$, so that $\alpha = \beta$.
\end{proof}

\begin{rmk}
    In a setting where \Cref{rem_subquotient-order} applies, the proposition remains true if the preservation properties are swapped.
\end{rmk}

\begin{Thm}\label{thm-loc}
    Let $f : \E \to \F$ be a localic geometric morphism, then $\pt(f) : \pt(\E) \to \pt(\F)$ is \emph{faithful}.
\end{Thm}
\begin{proof}
    Suppose that $\alpha, \beta$ are two $\sigma$-arrows $p \vuto{\sigma}q_s$ in $\pt(\E)$ which have the same image under $\pt(f)$.
    This means that $\alpha$ and $\beta$ are two natural transformations $p^*\Rightarrow \int_{s:\sigma}{q_s}^*$ that agree on objects of the form $f^*(X)$ with $X \in \F$.
    The functor $p^*$ is a left adjoint, so it preserves epimorphisms. Both $\int_{s:\sigma}$ and $({q_s}^*)$ preserve finite limits, so they preserve monomorphisms as well. \Cref{lem_localic} therefore applies and yields $\alpha = \beta$.
\end{proof}

The full picture relating faithful vu-functors and localic geometric morphisms is postponed to \Cref{table:loc}.

\section{Interlude: the merging strategy}\label{sec:merging}

Understanding surjections and hyperconnected maps from their vu-categories of points will be harder than for embeddings and localic morphisms. 
To deduce their vu-characterisations, we will reduce to the 0-dimensional case using \emph{localic slices} (see \Cref{topoetal}) and we will deduce the full 1-dimensional result by merging the 0-dimensional results along a \emph{semi-flat diagram of vu-sheaves} (see \Cref{semi-flat}).
In this section we settle the terminology and tools needed for this merging strategy.

\subsection{Reducing to the 0-dimension}\label{topoetal}

\begin{defn}\label{def:loc-slice}
    For a geometric morphism $f : \F\to\E$ and $E\in\E$, \emph{the localic slice} of $f$ above $E$ is the morphism of locales given by $\Lref{f/E} : \Lref{\F/f^*E} \to\Lref{\E/E}$.
    \[\begin{tikzcd}
    	{\F/f^*E} & {\E/E} \\
    	\F & \E
    	\arrow["{f/E}", from=1-1, to=1-2]
    	\arrow["{\text{ét}}"', from=1-1, to=2-1]
    	\arrow["\lrcorner"{anchor=center, pos=0.125}, draw=none, from=1-1, to=2-2]
    	\arrow["{\text{ét}}", from=1-2, to=2-2]
    	\arrow["f", from=2-1, to=2-2]
    \end{tikzcd}\] 
\end{defn}

\begin{rmk}\label{rmk:propclass}
    We propose to say that a class of geometric morphisms is \emph{propositional} if a geometric morphism belongs to the class if and only if all its localic slices belong to the class.
    As we will see later (\Cref{prop:surj-prop} and \Cref{shc-localic}), (sub)hyperconnected and surjective geometric morphisms form propositional classes of geometric morphisms. This is the crucial property that we will use to establish their vu-categorical characterisations, and we expect that the techniques presented in this paper can be extended to other propositional classes of geometric morphisms.
\end{rmk}

\begin{prop}\label{prop:loc-slice}
    For $F:\X\to\Y$ between bounded vu-categories and $B\in\Sh(\Y)$, the localic slice of $\Sh(F)$ above $B$ is given by 
    \[\shloc(\Tref{F^*\B}) \to \shloc(\Tref{\B}).\]
\end{prop}
\begin{proof}
    By \Cref{thm:etale}\ref{prop:slice-etale} the localic slice is given by $\Lref{\Sh(F^*\B)} \to \Lref{\Sh(\B)}$ and by \Cref{circular-square} this map can be rewritten as $\shloc(\Tref{F^*\B}) \to \shloc(\Tref{\B})$.
\end{proof}

We now give the vu-categorical counterpart of the localic slice.
\begin{defn}\label{def:etalements}
    For a vu-functor $F : \X\to\Y$ and $A\in\Sh(\Y)$, \emph{the topological étalement} of $F$ above $A$ is the vu-functor of vu-posets given by $\Tref{F/A} : \Tref{F^*\A} \to \Tref{\A}$. 
    \[\begin{tikzcd}
    	{F^*\A} & {\A} \\
    	\X & \Y
    	\arrow["{F/A}", from=1-1, to=1-2]
    	\arrow["{\text{ét}}"', from=1-1, to=2-1]
    	\arrow["\lrcorner"{anchor=center, pos=0.125}, draw=none, from=1-1, to=2-2]
    	\arrow["{\text{ét}}", from=1-2, to=2-2]
    	\arrow["F", from=2-1, to=2-2]
    \end{tikzcd}\]
\end{defn}

\begin{rmk}
    Note that if $\Y$ does not have enough vu-sheaves, the topological étalements of $F$ are not enough to recover interesting information on $F$.
\end{rmk}

\begin{prop}\label{prop:etalements}\leavevmode
    \begin{enumerate} 
        \item\label{prop:etalements:i}
        For $\X$ a bounded vu-category and $A\in\Sh(\X)$, the topological étalement of $\eta_\X:\X\to\pt(\Sh(\X))$ above $A$ is given by 
        \[\A\to\pt(\Sh(\A)).\]
        \item\label{prop:etalements:ii} For $f : \F\to\E$ a geometric morphism and $A\in\E$, the topological étalement of $\pt(f)$ above $\ev(A)$ is given by 
        \[\Tref{\pt(\F/f^*A)} \to \Tref{\pt(\E/A)}.\]
    \end{enumerate}
\end{prop}
\begin{proof}\leavevmode
    \begin{enumerate}
        \item By \Cref{thm:etale}\ref{prop:slice-etale}\ref{prop:point-etale}, we have 
        $\et{\ev(A)}\cong\pt(\Sh(\X)/A)\cong\pt(\Sh(\A))$.
        \[\begin{tikzcd}
            \A & {\pt(\Sh(\A))}\\
            \X & {\pt(\Sh(\X))}
            \arrow[from=1-1, to=1-2]
            \arrow["{\text{ét}}"', from=1-1, to=2-1]
            \arrow["\lrcorner"{anchor=center, pos=0.125}, draw=none, from=1-1, to=2-2]
            \arrow["{\text{ét}}", from=1-2, to=2-2]
            \arrow["{\eta_\X}", from=2-1, to=2-2]
        \end{tikzcd}\]
        
        \item By \Cref{thm:etale}\ref{prop:point-etale}, we have 
        $\et{\ev(A)} \cong \pt(\E/A)$ and $\et{\ev(f^*A)} \cong \pt(\F/f^*A)$.
        \[\begin{tikzcd}
        	\pt(\F/f^*A) & {\pt(\E/A)} \\
        	\pt(\E) & {\pt(\F)}
        	\arrow[from=1-1, to=1-2]
        	\arrow["{\text{ét}}"', from=1-1, to=2-1]
        	\arrow["\lrcorner"{anchor=center, pos=0.125}, draw=none, from=1-1, to=2-2]
        	\arrow["{\text{ét}}", from=1-2, to=2-2]
        	\arrow["{\pt(f)}", from=2-1, to=2-2]
        \end{tikzcd}\qedhere\]
    \end{enumerate}
\end{proof}

\subsection{Merging along a semi-flat diagram}\label{semi-flat}

We now introduce the notion of \emph{semi-flat diagram of vu-sheaves}. This notion captures exactly the condition needed on a diagram $[-] : \dsD\to\Sh(\X)$ so we can use it to merge together results shown on topological étalements above the $([D])_{D:\dsD}$. It will be a crucial tool for the proofs of \Cref{prop:Barr-categorified} and \Cref{prop:ultraretract-crux}.

\begin{nota}\label{nota:em}
    We fix $[-] : \dsD\to\Sh(\X)$ a small diagram of vu-sheaves over $\X$. For $p\in\X$, we denote by $\eM_{\dsD}(p)$ the category of elements of the functor $[-]_p : \dsD\to\Set$, and an element $\ell\in\eM_{\dsD}(p)$ will be denoted as a pair $\ell = (D_{\ell},\xi_{\ell})$ with $D_\ell\in\dsD$ and $\xi_\ell\in[D_\ell]_p$.
\end{nota}

\begin{defn}
    A diagram $[-] : \dsD\to\Sh(\X)$ is said \emph{semi-flat} if for any $p$ in $\X$ the posetal reflection of the category $\eM_{\dsD}(p)$ is codirected, \ie for any $\xi_0\in [D_0]_p$ and $\xi_1\in [D_1]_p$ there exists $\xi\in [D]$ with maps $D\rightrightarrows D_i$ mapping $\xi$ to $\xi_i$.
\end{defn}

The main motivation behind introducing the notion of semi-flatness is the following consequence of \Cref{lem:ext-final-uf}.

\begin{cor}\label{def:lambda}
    For $[-] : \dsD\to\Sh(\X)$ a semi-flat diagram, we can consider an \emph{initial ultrafilter} on the objects of $\eM_{\dsD}(p)$, that is, an ultrafilter $\lambda$ on pairs $(D,\xi)$, with $D\in\dsD$ and $\xi\in [D]_p$, such that for each $\xi_0\in [D_0]_p$, for $\lambda$-all $(D,\xi)$ there exists a $f : D\to D_0$ such that $[f]_p(\xi) = \xi_0$.
\end{cor}

We now give an easy way to produce semi-flat diagrams.

\begin{lem}
    If $\dsD$ has finite products and $[-] : \dsD\to\Sh(\X)$ preserves them, then $[-] : \dsD\to\Sh(\X)$ is semi-flat.
\end{lem}
\begin{proof}
    For $\xi_0\in[D_0]_p$ and $\xi_1\in[D_1]_p$, we consider $(\xi_0,\xi_1)\in[D_0\times D_1]_p$, and the two projections show that $(D_0\times D_1,(\xi_0,\xi_1))$ is smaller than both $(D_i,\xi_i)$ in the posetal reflection of $\eM_{\dsD}(p)$.
\end{proof}




\begin{cor}\label{prop:semi-flat-exist}
    Let $\dsD\subseteq \E$ be a small dense full sub-category of a topos $\E$ closed under products. Then the functor $ \dsD\to \Sh(\pt(\E))$ sending $D\in\dsD$ to the vu-sheaf $\ev(D)$ is a semi-flat diagram separating vu-arrows. 
    
    In particular, for any full sub-vu-category $\X\subseteq\pt(\E)$, one can find a small semi-flat diagram $[-] : \dsD\to\Sh(\X)$ of vu-sheaves of the form $\ev(E)$, that moreover separates vu-arrows.

\end{cor}

\begin{ex}
    Let $U\in\E$ be a prebound of $\E$ (\ie the corresponding geometric morphism $\E\to\SetO$ is localic). Then the functor $\mathds{N}^{\op} \to \Sh(\pt(\E))$ sending $n$ to the vu-sheaf $\ev(U^n)$, is a semi-flat diagram separating vu-arrows (where $\mathds{N}$ denotes the category of integers seen as finite sets and set-functions between them).
    One can think of $U_p$ as the underlying set of the model $p$. In that case, the posetal reflection of $\eM_{\dsD}(p)$ is given by the poset of finite subsets of $U_p$ (which can be seen as finite sets of $U$-parameters in $p$) ordered by reverse inclusion. Thus, an initial ultrafilter $\lambda$ as in \Cref{def:lambda} can be seen as a ‘finite set of $U$-parameters in $p$ that contains all the elements of $U_p$’.
\end{ex}

\section{Subhyperconnectedness}\label{sec:vu-full}

This section and the next one (\Cref{sec:surjectivity}) are fully independent, even though they follow a similar proof pattern.
 
In this section we focus on the notion of fullness of vu-functors, that happens to be more subtle than the notion of (fully) faithfulness.
A vu-full vu-functor is a vu-functor that is full \emph{up to thickening}. The main technical result is that if the codomain arises from the points of topos, vu-fullness can be checked on the topological étalements (\Cref{shc-vufull-cor}).

We then introduce the notion of subhyperconnected geometric morphism, which is the toposical counterpart of vu-full vu-functors. We first show that these two notions agree at the 0-dimensional level. We then establish the relation between the 1-dimensional objects (\Cref{thm:vufull}), using that these notions can be detected on the localic slices (\Cref{shc-localic}\ref{shc-slices}) and the topological étalements (\Cref{shc-vufull-cor}). 

Finally, we deduce by orthogonality the relation between localic geometric morphisms and faithful vu-functors (\Cref{thm-loc-bis}), and we conclude the section with a characterisation of vu-categories arising from a class of points of a topos as taut bounded vu-categories having enough sheaves (\Cref{locsobercaract}).


\subsection{Vu-fullness}

\begin{defn}
    A vu-functor $F:\X\to\Y$ is \emph{vu-full} if for $p$ and $(q_s)_{s:\sigma}$ in $\X$, any vu-arrow $g : F(p)\vuto{\sigma}{F(q_s)}$ can be thickened to be in the image of $F$, \ie there is $\phi : \tau\to\sigma$ and $f:p\vuto{\tau}{q_{\phi t}}$ such that $F(f) = \phi^*g$.
\end{defn}

In full generality, a vu-full and faithful vu-functor is not necessarily fully faithful; however it is the case if the vu-categories are assumed to be taut.

\begin{prop}\label{taut-ff}
    A vu-functor between taut vu-categories is fully faithful if and only if it is vu-full and faithful.
\end{prop}
\begin{proof}
    It is clear that a fully faithful vu-functor is faithful and vu-full. 
    Conversely, we take $F:\X\to\Y$ vu-full and faithful between taut vu-categories and $g : F(p)\vuto{\sigma}F(q_s)$, and we want to show that $g = F(f)$ for some $f:p\vuto{\sigma}q_s$.
    Vu-fullness gives us $\phi : \tau\to\sigma$ and $h:p\vuto{\tau}q_{\phi t}$ such that $F(h) = \phi^*g$.
    For any $\psi_1,\psi_2 : \lambda\rightrightarrows\tau$ equalising $\phi$, we have 
    \[F(\psi_1^*h) = \psi_1^*F(h) = \psi_1^*\phi^*g = \psi_2^*\phi^*g = \psi_2^*F(h) = F(\psi_2^*h)\]
    and so $\psi_1^*h = \psi_2^*h$ by faithfulness of $F$.
    Hence, the tautness hypothesis on $\X$ gives us some $f:p\vuto{\sigma}q_s$ such that $h = \phi^*f$, and we can then write
    \[\phi^*g = F(h) = F(\phi^*f) = \phi^*F(f),\]
    so that $F(f) = g$ by tautness of $\Y$.
\end{proof}

\begin{cor}\label{prop:vufull-posetal}
    A vu-functor between taut vu-posets is fully faithful if and only if it is vu-full.
\end{cor}
\begin{proof}
    A vu-functor between vu-posets is always faithful.
\end{proof}

\begin{prop}\label{prop:vufull}\leavevmode
    \begin{enumerate}
        \item\label{prop:vufull-topreflection} Vu-full functors are stable under topological reflections.
        \item\label{prop:vufull-pb} Vu-full functors are stable under pullbacks.
        \item\label{prop:vufull-topetalement} Vu-full functors are stable under topological étalements, \ie if $F : \X \to \Y$ is vu-full, then $\Tref{F/A} : \Tref{F^*\A} \to \Tref{\A}$ is fully faithful for each $A \in \Sh(\Y)$.
    \end{enumerate}
\end{prop}
\begin{proof}\leavevmode
    \begin{enumerate}
        \item[\pref{prop:vufull-topreflection}] We suppose that $F:\X\to\Y$ is vu-full and we want to show that $\Tref{F}:\Tref{\X}\to\Tref{\Y}$ is vu-full. We take $F(p)\vuleq{\sigma}F(q_s)$ in $\Tref{\Y}$ witnessed by some $g : F(p)\vuto{\tau}F(q_{\phi t})$ in $\Y$ for some $\phi : \tau\to\sigma$, the vu-fullness of $F$ gives us some $f : p\vuto{\lambda}q_{\phi \psi \ell}$ in $\X$ with $\psi:\lambda\to\tau$, showing that $p\vuleq{\sigma}q_s$ in $\Tref{\X}$.
        
        \item[\pref{prop:vufull-pb}] The proof is essentially the same as for usual 1-categories; we will use the explicit description of 2-pullbacks of vu-categories given in \cite[6.2(ii)]{Saa}.
        
        We take $F:\X\to\Z$ and $G:\Y\to\Z$, with $G$ vu-full, and we show that $\X\times_{\Z}\Y\to\X$ is again vu-full. Let $(\theta,x,y)$ and $(\theta_s,x_s,y_s)_{s:\sigma}$ be objects of $\X\times_{\Z}\Y$, \ie $\theta$ (resp. $\theta_s$) is an isomorphism from $F(x)$ to $G(y)$ (resp. from $F(x_s)$ to $G(y_s)$), and let $f: x\vuto{\sigma} x_s$.
        
        We consider 
        \[h \coloneqq \theta_s\circ_{\sigma} F(f) \circ_{\ptuf} \theta^{-1} : G(y)\vuto{\sigma}G(y_s),\]
        by vu-fullness of $G$ we can find some $\phi : \tau\to\sigma$ and some $g:y\vuto{\tau}y_{\phi t}$ such that $G(g) = \phi^*h$. And as $G(g)\circ\theta = (\theta_s)\circ F(f)$, we get a vu-arrow
        \[(\phi^*f,g) : (\theta,x,y)\vuto{\tau}(\theta_{\phi t},x_{\phi t},y_{\phi t})\]
        that is sent to $\phi^*f$ by $\X\times_{\Z}\Y\to\X$.

        \item[\pref{prop:vufull-topetalement}] Follows directly from the previous two. \qedhere
    \end{enumerate}
\end{proof}

Hence, if a vu-functor is vu-full, all its topological étalements are also vu-full (or equivalently by \Cref{prop:vufull-posetal}, fully faithful).
The rest of this section is dedicated to showing that the converse holds, assuming that the codomain is the vu-category of points of a topos. 

\begin{Thm}\label{shc-vufull}
    Let $F : \X\to\Y$ be a vu-functor and $[-]:\dsD\to\Sh(\Y)$ be a small semi-flat diagram of vu-sheaves separating vu-arrows.
    For $F : \X\to\Y$ to be vu-full it suffices that the topological étalements of $F$ above the $([D])_{D:\dsD}$ are fully faithful.
    
    
\end{Thm}

Using \Cref{prop:vufull}\pref{prop:vufull-topetalement} and \Cref{prop:semi-flat-exist} we can deduce the following corollary.

\begin{cor}\label{shc-vufull-cor}\leavevmode
    A vu-functor $\X\to\pt(\E)$ is vu-full if and only if all its topological étalements above the $\ev(E)$ with $E\in\E$ are fully faithful.
\end{cor}

The key ingredient to prove \Cref{shc-vufull} is the following lemma, showing how to use a semi-flat diagram to merge vu-arrows together.
\begin{lem}\label{merging-vumap}
    Let $[-] : \dsD\to\Sh(\X)$ be a semi-flat diagram, $p$ and $(q_s)_{s:\sigma}$ objects of $\X$, and maps $(g_D : [D]_p \to \int_{s:\sigma}[D]_{q_s})$ natural in $D\in\dsD$.
    
    To construct some $\phi : \tau\to\sigma$ and $f : p \vuto{\tau}{q_{\phi t}}$ such that
    \[\text{for all $\xi\in [D]_p$ }, f_*(\xi) = \phi^*g_D(\xi)\]
    it suffices to give for each $\xi\in [D]_p$ some $\phi_{\xi} : \tau_{\xi}\to\sigma$ and $f_{\xi} : p\vuto{\tau_{\xi}}{q_{\phi_{\xi} t}}$ such that
    \[{f_{\xi}}_*(\xi) = {\phi_{\xi}}^*(g_D(\xi)).\]
\end{lem}
\begin{proof}
    We suppose given for any $\xi\in [D]_p$ some $f_{\xi} : p\vuto{\tau_{\xi}}{q_{\phi_{\xi}t}}$ (with $\phi_{\xi} : \tau_{\xi}\to\sigma$) such that ${f_{\xi}}_*(\xi) = {\phi_{\xi}}^*(g_D(\xi))$.
    We then choose an initial ultrafilter $\lambda$ on $\eM_{\dsD}(p)$ as in \ref{def:lambda} and merge the above $f_{\xi}$ along it
    \[f : p \vuto[\delta]{\ell:\lambda}{(p\vuto[f_{\xi_\ell}]{\tau_{\xi_\ell}}{q_{\phi_{\xi_\ell} t}})}.\]
    For any $\ell_0 = (D_0,\xi_0)$ in $\eM_{\dsD}(p)$ (so $\xi_0\in [D_0]_p$), we have $f_*(\xi_0) = ({f_{\xi_\ell}}_*(\xi_0))_{\ell:\lambda}$, and for $\ell:\lambda$ with $\ell=(D,\xi)$, we can assume by initiality of $\lambda$ that $\ell\geq\ell_0$, and so there is a $\gamma : D\to D_0$ with $[\gamma]_p(\xi) = \xi_0$.
        
    The diagram chase below then shows that ${f_{\xi}}_*(\xi_0) = {\phi_{\xi}}^*g_{D_0}(\xi_0)$.

\[\begin{tikzcd}
	&& {\int_{\sigma}[D]_{q_s}} \\
	\xi & {[D]_p} && {\int_{\tau_{\xi}}[D]_{q_{\phi_{\xi} t}}} \\
	&& {\int_{\sigma}[D_0]_{q_s}} \\
	{\xi_0 = [\gamma]_p(\xi)} & {[D_0]_p} && {\int_{\tau_{\xi}}[D_0]_{q_{\phi_{\xi} t}}}
	\arrow["{{\phi_{\xi}}^*}"{description}, from=1-3, to=2-4]
	\arrow["{([\gamma]_{q_{s}})}"{description, pos=0.7}, from=1-3, to=3-3]
	\arrow["\in"{marking, allow upside down}, draw=none, from=2-1, to=2-2]
	\arrow[dashed, maps to, from=2-1, to=4-1]
	\arrow["{g_D}"{description}, from=2-2, to=1-3]
	\arrow["{[\gamma]_p}"{description, pos=0.7}, from=2-2, to=4-2]
	\arrow["{([\gamma]_{q_{\phi_{\xi} t}})}"{description, pos=0.7}, from=2-4, to=4-4]
	\arrow["{{\phi_{\xi}}^*}"{description}, from=3-3, to=4-4]
	\arrow["\in"{marking, allow upside down}, draw=none, from=4-1, to=4-2]
	\arrow["{g_{D_0}}"{description}, from=4-2, to=3-3]
	\arrow["{{f_{\xi}}_*}"{description, pos=0.7}, from=4-2, to=4-4]
    \arrow["{{f_{\xi}}_*}"{description, pos=0.7}, crossing over, from=2-2, to=2-4]
\end{tikzcd}\]
        \begin{align*}
            &{f_{\xi}}_*(\xi_0) \\ &= {f_{\xi}}_*([\gamma]_p(\xi))  \\
            &= [\gamma]_{q_{\phi_{\xi}(t)}}({f_{\xi}}_*(\xi)) \quad\text{naturality of $[\gamma]$} \\
            &= [\gamma]_{q_{\phi_{\xi}(t)}}({\phi_{\xi}}^*g_D(\xi)) \ \text{the top triangle at $\xi$ by construction of $f_{\xi}$} \\
            &= {\phi_{\xi}}^*[\gamma]_{q_{s}}(g_{D}(\xi)) \quad\text{reindexing of ultraproducts} \\
            &= {\phi_{\xi}}^*g_{D_0}([\gamma]_{p}(\xi)) \quad\text{naturality of the $(g_D)$ in $D$}  \\
            &= {\phi_{\xi}}^*g_{D_0}(\xi_0)   
        \end{align*}
        And denoting $\phi : (\ell:\lambda)\cdot\tau_{\xi_\ell} \stackrel{(\phi_{\xi_\ell})}\to \sigma$,
        \[f_*(\xi_0)  = ({f_{\xi_\ell}}_*(\xi_0))_{\ell:\lambda} = ({\phi_{\xi_\ell}}^*g_{D_0}(\xi_0))_{\ell:\lambda} = {\phi}^*g_{D_0}(\xi_0).\qedhere\]
\end{proof}

\begin{prop}\label{prop:Barr-categorified}
    Let $[-] : \dsD\to\Sh(\X)$ be a semi-flat diagram, $p$ and $(q_s)_{s:\sigma}$ objects of $\X$, and maps $(g_D : [D]_p \to \int_{s:\sigma}[D]_{q_s})$ natural in $D\in\dsD$.
    The following are equivalent:
    \begin{enumerate}
        \item\label{Barr-categorified:i} there exists $f : p \vuto{\tau}q_{\phi t}$, with $\phi : \tau\to\sigma$, such that $f_*(\xi) = \phi^*g_D(\xi)$ for all $\xi\in [D]_p$;
        \item\label{Barr-categorified:ii} for $\xi\in [D]_p$, $(p,\xi)\vuleq{\sigma} (q_s , \zeta_s)$ in $\Tref{\et{[D]}}$ where $(\zeta_s)_{s:\sigma} \coloneqq g_D(\xi)$.
    \end{enumerate}
\end{prop}
\begin{proof}\leavevmode
    \begin{enumerate}
        \item[\pref{Barr-categorified:i} $\Rightarrow$ \pref{Barr-categorified:ii}] By hypothesis we have an $f : p \vuto{\tau}{q_{\phi t}}$ such that $f_*(\xi) = (\zeta_{\phi t})_{t:\tau}$, so we have a vu-arrow $(p,\xi)\vuto{}{(q_{\phi t},\zeta_{\phi t})}$ in $\et{[D]}$, and \pref{Barr-categorified:ii} follows.
        
        \item[\pref{Barr-categorified:ii} $\Rightarrow$ \pref{Barr-categorified:i}] Follows from \Cref{merging-vumap}.\qedhere
    \end{enumerate}
\end{proof}

\begin{rmk}\label{rmk:Barr-categorified}
    \Cref{cor:Barr-taut} ensures that \pref{Barr-categorified:ii} is equivalent to
    \begin{enumerate}\setcounter{enumi}{2}
        \item \label{Barr-categorified:iii} for any $A\subseteq [D]$ sub-vu-sheaf, $g_D$ restricts to a map $A_p \to \int_{s:\sigma}A_{q_s}$,
    \end{enumerate}
    and we recover \Cref{prop:Barr} by looking at the implication \pref{Barr-categorified:iii}$\Rightarrow$\pref{Barr-categorified:i} for the terminal diagram $1\to\Sh(\X)$, that is, the diagram pointing on the terminal vu-sheaf.
    In this case of the terminal diagram, we can give the following logical interpretation: if any closed formula satisfied by $p$ is also satisfied by the $(q_s)_{s:\sigma}$, then there exists a homomorphism from $p$ into some thickening of $(q_s)_{s:\sigma}$.
    The case of a general diagram can be seen as a generalisation allowing parameters: if a $\dsD$-diagram of set-functions from $p$ to $(q_s)_{s:\sigma}$ defined on the underlying sorts prescribed by $\dsD$ respects all the formulas with parameters, then it can be thickened to become a homomorphism of models.
\end{rmk}

        
        
    

We can finally prove \Cref{shc-vufull}.
\begin{proof}[Proof of \Cref{shc-vufull}]

    Let $F : \X\to\Y$ be a vu-functor and $[-] : \dsD\to\Sh(\Y)$ a small semi-flat diagram separating vu-arrows such that the vu-poset morphism $\Tref{F/[D]} : \Tref{F^*\et{[D]}} \to \Tref{\et{[D]}}$ is fully faithful for each $D\in\dsD$.
    %
    We want to show that $F$ is vu-full. We take some $g : F(p)\vuto{\sigma}{F(q_s)}$, and we want to construct a $f:p\vuto{\tau}{q_{\phi t}}$ with $\phi:\tau\to\sigma$ such that $F(f) = \phi^*g$. 
    
    This $f$ will be constructed by using \Cref{prop:Barr-categorified} with the diagram
    \[\llbracket - \rrbracket \coloneqq F^*[-] : \dsD\to\Sh(\X)\]
    which is a semi-flat diagram (because $F^* : \Sh(\Y)\to\Sh(\X)$ is lex).
    We show that the natural maps 
    \[(g_D : \llbracket D \rrbracket_p \to\textstyle\int_{s:\sigma}\llbracket D \rrbracket_{q_s})_{D:\dsD}\]
    given by pushing along $g$, satisfy \Cref{prop:Barr-categorified}\pref{Barr-categorified:ii}, \ie that for all $\xi\in \llbracket D \rrbracket_p$
    \[(p,\xi)\vuleq{\sigma}(q_s,\zeta_s) \text{ in } \Tref{\et{\llbracket D \rrbracket}} \text{ , with }(\zeta_s)_{s:\sigma} \coloneqq g_*(\xi).\]
    
    These points are mapped by the fully faithful map $\Tref{\et{\llbracket D \rrbracket}}\to \Tref{\et{[D]}}$ to respectively $(F(p),\xi)$ and $(F(q_s),\zeta_s)$, and we have the vu-inequality $(F(p),\xi)\vuleq{\sigma}(F(q_s),\zeta_s)$ in $\Tref{\et{[D]}}$ as $g$ induces a vu-arrow $(F(p),\xi)\vuto{\sigma}(F(q_s),\zeta_s)$ in $\et{[D]}$. 
    
    Hence these $(g_D)$ satisfy \ref{prop:Barr-categorified}\pref{Barr-categorified:ii} and \ref{prop:Barr-categorified}\pref{Barr-categorified:i} gives us some vu-arrow $f : p\vuto{\tau}(q_{\phi t})$ with $\phi : \tau\to\sigma$ such that 
    \[f_*(\xi) = \phi^*g_D(\xi) \text{ , for all }  \xi\in \llbracket D \rrbracket_p,\]
    \ie $F(f)_*(\xi) = (\phi^*g)_*(\xi)$ for all $\xi\in [D]_{F(p)}$. And $\dsD$ separating vu-arrows let us conclude that $F(f) = \phi^*g$. 
\end{proof}

\subsection{Subhyperconnectedness}\label{sec:shc}

In \cite[A4.6.10]{johnstone:elephant}, Johnstone characterises full functors between small categories by their presheaf toposes: a functor is full if and only if the geometric morphism it induces is a composite of a hyperconnected map followed by an embedding. We call such geometric morphisms \emph{subhyperconnected}.

\begin{defn}
    A geometric morphism $f:\F\to\E$ is said to be \emph{subhyperconnected} if it can be written as a composite $f = i \circ h$ of $h$ hyperconnected followed by $i$ an embedding.
\end{defn}

We recall the following stability properties of subhyperconnected geometric morphisms, analogous to the one proved for vu-full vu-functors.

\begin{prop}\label{shc-localic}\leavevmode 
    \begin{enumerate}
        \item\label{shc-localic-i} Subhyperconnected geometric morphisms are stable under localic reflections, and the localic reflection of a subhyperconnected map is an embedding of locales.
        \item\label{shc-ref} Subhyperconnected morphisms are stable under pullback.
        \item\label{shc-slices} A geometric morphism $f : \F\to\E$ is subhyperconnected if and only if its localic slices are embeddings of locales, \ie $\Lref{f/E}  : \Lref{\F/f^*E} \mono\Lref{\E/E}$ is an embedding of locales for any $E\in\E$.
    \end{enumerate}
\end{prop}
\begin{proof}\leavevmode
    \begin{enumerate}
        \item The localic reflection preserves embeddings \cite[A4.6.12]{johnstone:elephant} and sends hyperconnected morphisms to equivalences \cite[A4.6.6]{johnstone:elephant}.
        \item Hyperconnected geometric morphisms are stable under pullbacks \cite[B3.3.7]{johnstone:elephant}, as well as embeddings \cite[A4.5.14(e)]{johnstone:elephant}.
        \item For the non-trivial direction of the last claim, we consider $f = i \circ p$ the surjection--embedding factorisation of $f$
        \[f : \F\stackrel{p}\twoheadrightarrow\G\stackrel{i}\hookrightarrow\E\]
        and we show that $p$ is hyperconnected. 
        Let $G\in\G$, we can write the following factorisation, where $E\coloneqq i_*G$ (and so $G = i^*E$ as $i$ is an embedding).
        \[\begin{tikzcd}
        	{\Lref{\F/p^*G} } && {\Lref{\E/i_*G} } \\
        	& {\Lref{\G/G} }
        	\arrow["{\Lref{f/E} }", from=1-1, to=1-3]
        	\arrow["{\Lref{p/G} }"', from=1-1, to=2-2]
        	\arrow["{\Lref{i/E} }"', from=2-2, to=1-3]
        \end{tikzcd}\]
        By \cite[A4.6.6(vi)]{johnstone:elephant} it is enough to show that $\Lref{p/G}  :\Lref{\F/p^*G} \to\Lref{\G/G} $ induces an equivalence of locales. The morphism of locales $\Lref{p/G} $ is surjective as the localic slice of a surjective geometric morphism (see \Cref{prop:surj-prop} below); moreover, the assumption that $\Lref{f/E} $ is an embedding makes $\Lref{p/G} $ one too by the cancellation properties of embeddings. As a result $\Lref{p/G} $ is an equivalence as a geometric morphism which is both an embedding and a surjection \cite[A4.2.11]{johnstone:elephant}. \qedhere

        %
        
    \end{enumerate}
\end{proof}

We can finally compare subhyperconnected geometric morphisms and vu-full vu-functors. We first begin with a lemma establishing the 0-dimensional case.
\begin{lem}\label{lem:full-0dim}
    A continuous map of topological spaces is fully faithful (as a morphism of vu-posets) if and only if the morphism of locales it induces is an embedding.
    Conversely, a morphism between locales with enough points is an embedding if and only if it induces a fully faithful vu-functor on its points. 
\end{lem}
\begin{proof}
    An embedding of locales, or an embedding of topological spaces, induces a fully faithful functor on the points by \Cref{cor:emb-ff}. Conversely, if $f : X\to Y$ is a fully faithful vu-functor between topological spaces seen as vu-posets, then up to equivalence we can suppose that $f$ is an inclusion of a sub-vu-poset of $Y$, and \Cref{prop:opens-are-upsets} shows that the opens of $X$ are exactly the restriction of the opens of $Y$.
\end{proof}

\begin{Thm}\leavevmode\label{thm:vufull}
    \begin{enumerate}
        \item\label{thm:vufull:iii} Let $F : \X\to\Y$ be a vu-full vu-functor between bounded vu-categories, then $\Sh(F)$ is subhyperconnected. 
        \item\label{thm:vufull:i} Let $\X$ be a bounded vu-category, then $\eta_{\X} : \X\to\pt(\Sh(\X))$ is vu-full.
        \item\label{thm:vufull:ii} Let $f: \F\to\E$ be a subhyperconnected geometric morphism, then $\pt(f)$ is vu-full.
    \end{enumerate}
\end{Thm}
\begin{proof}\leavevmode
    \begin{enumerate}
        \item We show the hypothesis of \Cref{shc-localic}\pref{shc-slices}. We take $B\in\Sh(\Y)$, by \Cref{prop:loc-slice} we have to show that 
        \[\shloc(\Tref{F^*\B}) \to \shloc(\Tref{\B})\] is an embedding of locales.
        
        By \Cref{prop:vufull}\pref{prop:vufull-topetalement} $\Tref{F^*\B} \to \Tref{\B}$ is fully faithful, and we conclude by \Cref{lem:full-0dim}.

        \item We show the hypothesis of \Cref{shc-vufull-cor}. We take $A\in\Sh(\X)$, by \Cref{prop:etalements}\ref{prop:etalements:i} we have to show that 
        \[\Tref{\A}\to\Tref{\pt(\Sh(\A))}\]
        is fully faithful.

        We have
        \[\begin{tikzcd}
        	{\Tref{\A}} & {\Tref{\ptvu(\shvu(\A))}} \\
        	{\ptloc(\shloc(\Tref{\A})))} & {\ptloc(\Lref{\shvu(\A)})}
        	\arrow[from=1-1, to=1-2]
        	\arrow[hook, from=1-1, to=2-1]
        	\arrow[hook, from=1-2, to=2-2]
        	\arrow["{\cong}", from=2-1, to=2-2]
        \end{tikzcd}\]
        where $\ptloc(\shloc(\Tref{\A})))\cong \ptloc(\Lref{\shvu(\A)})$ by \Cref{circular-square}, 
        $\Tref{\ptvu(\shvu(\A))}\mono\ptloc(\Lref{\shvu(\A)})$ is fully faithful by \Cref{prop:rpt=ptr},
        $\Tref{\A}\to\ptloc(\shloc(\Tref{\A})))$ is fully faithful by \Cref{lem:full-0dim},
        and we conclude by cancellative property of fully faithful vu-functors.
        
        \item We will show that $\pt(f)$ satisfies \ref{shc-vufull-cor}. We take $A\in\E$, by \Cref{prop:etalements}\ref{prop:etalements:ii} we have to show that
        \[\Tref{\pt(\F/f^*A)} \to \Tref{\pt(\E/A)}\]
        is fully faithful.

        We have
        \[\begin{tikzcd}
        	{\Tref{\pt(\F/f^*A)}} & {\Tref{\pt(\E/A)}} \\
        	{\pt(\Lref{\F/f^*A})} & {\pt(\Lref{\E/A})}
        	\arrow[from=1-1, to=1-2]
        	\arrow[hook, from=1-1, to=2-1]
        	\arrow[hook, from=1-2, to=2-2]
        	\arrow[hook, from=2-1, to=2-2]
        \end{tikzcd}\]
        where $\Tref{\pt(\F/f^*A)}\mono\pt(\Lref{\F/f^*A})$ and $\Tref{\pt(\E/A)}\mono\pt(\Lref{\E/A})$ are fully faithful by \Cref{prop:rpt=ptr},
        $\pt(\Lref{\F/f^*A} )\to \pt(\Lref{\E/A})$ is fully faithful by \Cref{shc-localic}\pref{shc-ref} and \Cref{lem:full-0dim},
        and we conclude by cancellative property of fully faithful vu-functors.\qedhere
     \end{enumerate}
\end{proof}

We can fill one more line of (\ref{tableau}). 

\begin{cor}\label{table:shc}
    \

    \centerline{
    \begin{tabular}{r l||C C|C C|C C|}
    & & \multicolumn{2}{c|}{$f\squig\pt(f)$}&
        \multicolumn{2}{c|}{$\Sh(F)\squig F$}&
        \multicolumn{2}{c|}{$f\squig F$}\\
    \cline{3-8}
    geometric morphisms & vu-functors & $\Rightarrow$ & $\Leftarrow$ & $\Rightarrow$ & $\Leftarrow$ & $\Rightarrow$ & $\Leftarrow$ \\
    \cline{1-8}
    subhyperconnected & vu-full
        & \checkmark & \checkmark
        & $\crossmark$ & \checkmark 
        & \checkmark & \checkmark
    \end{tabular}
    }
\end{cor}
\begin{proof}
    All the \checkmark can be deduced from Theorems \ref{thm:reconstruction} and \ref{thm:vufull}. 
    For the counter-examples, we consider $\X$ from \Cref{counterex}. 
    If $F : [1] \to \X$ is a vu-functor pointing on any $\ptuf$-arrow of $\X$, then $F$ is not vu-full, although $\Sh(F)$ is an equivalence of toposes and so in particular is subhyperconnected.
\end{proof}

\subsection{Localicness}

By orthogonality we can now show the remaining relations between localic geometric morphisms and faithful vu-functors.
\begin{Thm}\label{thm-loc-bis}\leavevmode
    \begin{enumerate}
        \item\label{thm:loc:i} Let $f$ be a geometric morphism between toposes with enough, if $\pt(f)$ is faithful then $f$ is localic.
        \item\label{thm:loc:iii} Let $\X$ be a bounded vu-category, if $F : \X \to \pt(\E)$ is faithful then its adjunct $f : \Sh(\X) \to \E$ is localic. 
    \end{enumerate}
\end{Thm}
\begin{proof}\leavevmode
    \begin{enumerate}
        \item We consider $f = \ell\circ h$ the hyperconnected--localic factorisation of $f$ and we show that $h$ is an embedding. The assumption that $\pt(f)$ is faithful implies that $\pt(h)$ is also faithful, but by \Cref{thm:vufull}\ref{thm:vufull:ii} $\pt(h)$ is also vu-full, hence $\pt(h)$ is fully faithful by \Cref{taut-ff} and \Cref{prop:sober-taut}, and so $h$ is an embedding by \Cref{table:emb}.
        \item By \ref{thm:loc:i} it is enough to show that $\pt(\Sh(\X))\to\pt(\E)$ is faithful, and that can be deduced from the composition
        \[\X\to\pt(\Sh(\X))\to\pt(\E)\]
        with $\X\to\pt(\E)$ being faithful and $\X\to\pt(\Sh(\X))$ being vu-full and bijective on objects. \qedhere
    \end{enumerate}
\end{proof}

And we fill the table (\ref{tableau}). 

\begin{cor}\label{table:loc}
    \

    \centerline{
    \begin{tabular}{r l||C C|C C|C C|}
    & & \multicolumn{2}{c|}{$f\squig\pt(f)$}&
        \multicolumn{2}{c|}{$\Sh(F)\squig F$}&
        \multicolumn{2}{c|}{$f\squig F$}\\
    \cline{3-8}
    geometric morphisms & vu-functors & $\Rightarrow$ & $\Leftarrow$ & $\Rightarrow$ & $\Leftarrow$ & $\Rightarrow$ & $\Leftarrow$ \\
    \cline{1-8}
    localic & faithful
        & \checkmark & \checkmark 
        & $\crossmark$ & ?
        & $\crossmark$ & \checkmark
    \end{tabular}
    }
\end{cor}
\begin{proof}
    All the \checkmark can be deduced from Theorems \ref{thm:reconstruction}, \ref{thm-loc}, and \ref{thm-loc-bis}. 
    For the counter-examples, we consider $\X$ from \Cref{counterex}, then the unit $F : \X\to\ [1]$ is not faithful, although $f : \Sh(\X)\to\Set^{[1]}$ is an equivalence of toposes and so in particular is localic.
\end{proof}

We can deduce the following characterisation, identifying vu-categories arising from a class of points of a topos and taut bounded vu-categories that have enough sheaves.
\begin{cor}\label{locsobercaract}
    Let $\X$ be a taut bounded vu-category. The following are equivalent:
    \begin{enumerate}
        \item \label{locsobercaract:i} the unit $\X\to\pt(\Sh(\X))$ is fully faithful;
        \item\label{locsobercaract:ii} $\X$ fully embeds into a vu-category of the form $\pt(\E)$;
        \item \label{locsobercaract:iii}$\X$ admits a faithful vu-functor into a vu-category of the form $\pt(\E)$;
        \item \label{locsobercaract:iv}$\X$ has enough sheaves, that is, $\X\to\pt(\Sh(\X))$ is faithful.
    \end{enumerate}
    These conditions are moreover equivalent to:
    \begin{enumerate}[start=5]
        \item \label{locsobercaract:v} $\X$ fully embeds into the vu-category of points of a presheaf topos;
        \item\label{locsobercaract:vi} $\X$ admits a faithful vu-functor into the vu-category of sets.
    \end{enumerate}
    
\end{cor}
\begin{proof}
    The implications \ref{locsobercaract:i}$\Rightarrow$\ref{locsobercaract:ii} and \ref{locsobercaract:ii}$\Rightarrow$\ref{locsobercaract:iii} are trivial. 
    For \ref{locsobercaract:iii}$\Rightarrow$\ref{locsobercaract:iv} we factorise $\X\to\pt(\E)$ by the unit
    \[\X\to\pt(\Sh(\X))\to\pt(\E),\]
    so that $\X\to\pt(\E)$ faithful implies $\X\to\pt(\Sh(\X))$ faithful.
    And \ref{locsobercaract:iv}$\Rightarrow$\ref{locsobercaract:i} follows easily from \Cref{thm:vufull}\ref{thm:vufull:i}, \Cref{taut-ff}, and \Cref{prop:sober-taut}.

    Finally, \ref{locsobercaract:v} and \ref{locsobercaract:vi} follow, respectively, from the fact that any topos embeds into a presheaf topos, and that any topos is localic over $\SetO$. 
\end{proof}

\section{Surjectivity}\label{sec:surjectivity}

This section is dedicated to the vu-categorical counterpart of geometric surjections.
We begin with the easy case of when a set of points of a topological space is separating (\Cref{prop:subcl-sep-topological-case}). We then categorify this result, introducing the crucial notion of \emph{vu-retractions} (\Cref{def:vu-retract}). Finally, we define \emph{vu-subdense} vu-functors as a vu-counterpart to geometric surjections, and, using techniques similar to those in \Cref{sec:vu-full}, we show how they relate to each other.



\subsection{Subclosure in vu-posets}\label{sec:subcl}

\begin{defn}\label{def:separating-in-vu-pos}
    A collection $A$ of points of a vu-poset (or a topological space) $X$ is said \emph{separating} if any two opens $U \subseteq V$ that agree on $A$ are equal.
\end{defn}

If one takes a logical point of view, in which the points of $X$ are would-be models of some propositional theory and the opens are would-be formulas of that theory, then a separating collection of models is one that can decide entailment of formulas.
Perhaps surprisingly, since the notion of separation deals directly with the duality between points and opens, it is possible to characterise separation entirely in terms of points. A version of this story, for topological spaces, already appears in \cite{johnstone-1980:open-maps}, whence the next notion originates.

In \cite{johnstone-1980:open-maps}, Johnstone gives a full answer to question of when a class of points of a topological space is separating: he defines the \emph{subclosure} of $A\subseteq X$ as the set of points $p\in X$ such that every locally closed set (intersection of a closed set and an open) containing $p$ intersects $A$, and then shows that $A$ is separating if and only if its subclosure is all of $X$.
%

We want to extend this notion of subclosure to the setting of vu-posets and then vu-categories, for that, we first give an alternative description of the subclosure in terms of ultraconvergence.

\begin{prop}\label{prop:subclosure-topology}
    Let $A$ be a collection of points of a vu-poset (or a topological space) $X$. For $p$ a point of $X$, the following are equivalent:
    \begin{enumerate}
        \item every locally closed subspace (intersection of an open and a closed subspace, see \ref{def:open}) containing $p$ meets $A$; \label{cond:subclosure-johnstone}
        \item $p$ belongs to $\cl(A \cap \cl\{p\})$; \label{cond:subclosure-closure}
        \item there exists an ultrafamily $(a_s)_{s:\sigma}$ of points in $A$ and a $\sigma$-family of ultrafilters $(\tau_s)_{s:\sigma}$ such that $p \vuleq{\sigma} a_s$ and $a_s \vuleq{\tau_s} p$ for all $s:\sigma$. If moreover $X$ is taut, we can assume all the $\tau_s$ are trivial. \label{cond:subclosure-vu-arrow}
    \end{enumerate}
\end{prop}
\begin{proof}\leavevmode
    \begin{enumerate}
        \item[$\pref{cond:subclosure-johnstone}\Leftrightarrow\pref{cond:subclosure-closure}$] By \Cref{lem:smallest-closedset}, $\cl\{p\}$ is the smallest closed subspace containing $p$, so \pref{cond:subclosure-johnstone} is equivalent to
        \[\text{for all open } U\text{ containing }p,\ U\cap\cl\{p\}\text{ intersects } A,\]
        which by \ref{lem:closure} means $p\in\cl(A\cap\cl\{p\})$.
        \item[$\pref{cond:subclosure-closure}\Leftrightarrow\pref{cond:subclosure-vu-arrow}$]
        Follows directly from \Cref{def:open}. \qedhere
    \end{enumerate}
\end{proof}

\begin{defn}\label{def:poset-subcl}
    If the above condition is satisfied we say that \emph{$p$ is in the subclosure of $A$}, and we denote by $\subcl(A)\subseteq X$ the collection of $p$ that are in the subclosure of $A$.
    Concretely, $p\in\subcl(A)$ if there exists a $\sigma$-family $(a_s)_{s:\sigma}$ of points in $A$ and a $\sigma$-family of ultrafilters $(\tau_s)_{s:\sigma}$ such that $p \vuleq{\sigma} a_s$ and $a_s \vuleq{\tau_s} p$ for $s:\sigma$.
\end{defn}

\begin{rmk}
    The condition \pref{cond:subclosure-johnstone} shows that if $X$ is the vu-poset of points of a topological space we recover the definition of Johnstone \cite[before Lem.~2.1]{johnstone-1980:open-maps}. The condition \pref{cond:subclosure-vu-arrow} is an equivalent formulation from the vu-posetal point of view, it is the one we will use to categorify the notion of subclosure. 
\end{rmk}

We can now show an analogue of \cite[Lem.~2.1]{johnstone-1980:open-maps}. 

\begin{prop}\label{prop:subcl-sep-topological-case}
    A collection of points $A$ of a vu-poset $X$ is separating if and only if its subclosure is $X$.
\end{prop}
\begin{proof}
    We suppose that $\subcl(A) = X$ and we show that $A$ is separating. Let $U$ and $V$ be two opens which agree on $A$, and pick $p \in U$, we want to show $p\in V$.
    By assumption, there exists $(a_s)_{s:\sigma}$ in $A$ such that $p \vuleq{\sigma} a_s$ and $a_s \vuleq{\tau_s} p$.
    From $p \vuleq{\sigma} a_s$ and $p \in U$, it follows that for $s:\sigma$, we have $a_s \in U$ and so $a_s \in V$.
    In particular there exists at least one $s_0$ such that $a_{s_0} \in V$ and $a_{s_0} \vuleq{\tau_{s_0}} p$, which guarantees $p \in V$. This shows $U \subseteq V$, and by symmetry the reverse also holds.

    Conversely, we suppose that $A$ is separating in $X$. By \Cref{lem:closure}, it is enough to show that any open $U$ containing $p$ intersects $A\cap \cl\{p\}$.
    As $U \setminus \cl\{p\}$ is an open strictly contained in $U$, it must be separated from $U$ by some $a\in A$. In other words, there exists $a \in A$ such that $a \in U$ and $a\in\cl\{p\}$, hence $U$ intersects $A\cap \cl\{p\}$ as wanted. 
\end{proof}



\subsection{Subclosure categorified}\label{sec:subcl-cat}

We now categorify the notion of separating set of points and of subclosure. The end-goal is to prove \Cref{prop:subcl-sep-categorified} that is a categorification of \Cref{prop:subcl-sep-topological-case}.

The following definition extends \Cref{def:separating-in-vu-pos} to vu-categories.

\begin{defn}\label{def:separating-in-vu-cat}
    A collection of points $\dsA$ of a vu-category $\X$ is \emph{separating} if every map of sheaves $f: A\to B$ such that all $(f_a : A_a\to B_a)_{a:\dsA}$ are isomorphisms is itself an isomorphism, \ie if the evaluation functors $(\ev_a : \Sh(\X)\to\Set)_{a:\dsA}$ are jointly conservative.
\end{defn}

Again, we would like to give a sheaf-less characterisation of separating families of points.
The following particular case can be found in the literature.

\begin{prop} \label{prop:presheaf-separating-retract}
    Let $\C$ be a small category seen as a vu-category (see \cite[4.2(vi)]{Saa}),
    then $\dsA\subseteq\Ob(\C)$ is separating if and only if every object of $\C$ is a retract of an object of $\dsA$.
\end{prop}
\begin{proof}
    The category of vu-sheaves above $\C$ is precisely the presheaf topos $\Set^{\C}$.
    Hence, this follows from \cite[A4.2.7(b)]{johnstone:elephant},
    by noticing that a set of objects of $\C$ separates $\Sh(\C)$ if and only if it induces a separating set of points of the presheaf topos $\Set^{\C}$.
\end{proof}

The above condition \ref{prop:subclosure-topology} \pref{cond:subclosure-vu-arrow} for vu-posets strongly suggests the retract pattern of the above categorical case \ref{prop:presheaf-separating-retract}. One would like to interpret \Cref{prop:subclosure-topology} \pref{cond:subclosure-vu-arrow} as saying that the point $p$ is a ‘vu-retract’ of the family $(a_s)_{s:\sigma}$.

\begin{defn}\label{def:vu-retract}
    Let $\X$ be any vu-category. A point $p$ is a \emph{vu-retract} of $(a_s)_{s:\sigma}$ if there are vu-arrows $i : p \vuto{\sigma} a_s$ and $(r_s : a_s\vuto{\tau_s} p)_{s:\sigma}$ such that $r_s\circ_{\sigma} i : p \vuto{{\sigma}\cdot\tau_s} p$ is equal to the diagonal $\delta : p\vuto{{\sigma}\cdot\tau_s}p$. The pair $(i,(r_s)_{s:\sigma})$ will be called a \emph{vu-retraction}.
\end{defn}


\begin{ex}\leavevmode\label{ex:vu-retract}
    We can recover our two motivating examples.
    \begin{enumerate}
        \item In a topological space $T$, a point $p$ is a vu-retract of points in $A \subseteq X$ precisely if it belongs to $\subcl(A)$ (see \Cref{def:poset-subcl}).
        \item In a 1-category $\C$ seen as a vu-category, if $p$ is a retract of $a$ in the usual sense, then $p$ is a vu-retract of the constant family $(a)_*$.
    \end{enumerate}
    Here are some more interesting, and maybe unexpected, examples.
    \begin{enumerate}[start=3]
        \item In an ultracategory seen as a vu-category (see \cite[4.2.(vii)]{Saa}), any object $A$ is an ultraretract of its ultrapower $(A^{\sigma})_{\ptuf}$, with $i : A \vuto{\ptuf} A^{\sigma}$ given by the diagonal $A\to A^{\sigma}$ and $r : A^{\sigma}\vuto{\sigma}A$ given by the identity $A^{\sigma} \to A^{\sigma}$. Note that this example shows the need for the $(\tau_s)$ in the definition of vu-retraction. 
        
        \item Perhaps surprisingly, the notion of vu-retract is powerful enough to see filtered colimits. Let $(a_i)_{i:I}$ be a filtered diagram in some vu-category $\X$, with colimit $p$ in the vu-categorical sense (\ie the universal property holds for vu-arrows).
        Then, for any final ultrafilter $\sigma$ on $I$ (see \Cref{def:final-uf}), $p$ is a vu-retract of $(a_i)_{i:\sigma}$,
        the $(r_s : a_s\vuto{\ptuf} p)_{s:\sigma}$ come from the cocone defining $p$, and the vu-arrow $i : p \vuto{\sigma} a_s$ is given by the universal property of $p$ (consider the composite of the diagonal of $a_i\vuto{\sigma}a_i$ with the $(a_i \vuto{\ptuf} a_s)_{s:\sigma}$).
    \end{enumerate}
\end{ex}

\begin{rmk}
    We believe that, assuming $\X$ taut, one can fully characterize such filtered colimits (in the above vu-categorical sense) by means of vu-retracts. This would imply that in the context of taut vu-categories, filtered colimits are \emph{absolute}, thereby explaining why vu-functors between vu-categories of points of toposes automatically preserve filtered colimits. Note that a particular case of this fact has already been observed in \cite[5.3.4]{Lurie}.
\end{rmk}

\begin{rmk}
    It is worth noting that a characterisation of soberness by nets for topological spaces has been given in \cite{Nets}. The author uses nets instead of filters, but we can make the following analogies: an ‘observative net’ corresponds to some yet-to-be-defined notion of ‘vu-idempotent’, and ‘strongly convergence’ corresponds to some ‘splitting’ of it. Moreover, the ‘b-topology’ is the topology generated by the locally closed sets, connecting the ‘b-closure’ to the subclosure.
\end{rmk}

\begin{defn}\label{def:cat-subcl}
    The \emph{subclosure} of a collection $\dsA$ of objects of a vu-category $\X$ is the full sub-vu-category consisting of the points which are vu-retracts of points in $\dsA$. We will denote it by $\subcl_\X(\dsA)\subseteq\X$, or just $\subcl(\dsA)$ if clear from context.
\end{defn}

\begin{rmk}\label{rmk:subclosure-Tref}
    The above notion is a categorification of \Cref{def:poset-subcl} in the sense that for vu-posets we recover the already defined notion. 
    Not only is it true that $\subcl_X(A) = \subcl_{\Tinj(X)}(A)$, as we have noted, but also $\subcl_{\X}(\dsA) \subseteq \subcl_{\Tref{\X}}(\dsA)$, although the inclusion is strict in general as we ask for an additional coherence.
\end{rmk}


We can now state the categorified analogue of \Cref{prop:subcl-sep-topological-case}: a collection of points of a vu-category is separating if and only if it generates it by vu-retracts.
We start with the more straightforward direction.
\begin{prop}\label{vu-retract-direction-facile}
    If all points of $\X$ are vu-retracts of points in $\dsA$ then $\dsA$ is separating.
\end{prop}
\begin{proof}
    The category $\Sh(\X)$ is a pretopos, so to show that $(\ev_a)_{a:\dsA}:\Sh(\X)\to\Set^{\dsA}$ is conservative it is enough to show that it is faithful.
    Let $f$ and $g$ be two parallel maps $A \rightrightarrows B$ in $\Sh(\X)$ which agree on all points $a\in \dsA$, we need to show that $f = g$. For $p \in \X$, we choose a vu-retraction $i : p \vuto{\sigma} a_s$ and $(r_s : a_s \vuto{\tau_s} p)_{s:\sigma}$ with the $(a_s)_{s:\sigma}$ in $\dsA$ and we consider the following diagram of sets.
\[\begin{tikzcd}
	{A_p} & {B_p} & \\
	{\int_{\sigma}A_{a_s}} & {\int_{\sigma}B_{a_s}} & {\int_{\sigma}\int_{\tau_s}B_p}
	\arrow["{f_p}"{description}, curve={height=-6pt}, dashed, from=1-1, to=1-2]
	\arrow["{g_p}"{description}, curve={height=6pt}, dashed, from=1-1, to=1-2]
	\arrow["{i_*}"', from=1-1, to=2-1]
	\arrow["{i_*}"', from=1-2, to=2-2]
	\arrow["\delta", hook, from=1-2, to=2-3]
	\arrow[""{name=0, anchor=center, inner sep=0}, "{(f_{a_s})}", curve={height=-6pt}, from=2-1, to=2-2]
	\arrow[""{name=1, anchor=center, inner sep=0}, "{(g_{a_s})}"', curve={height=6pt}, from=2-1, to=2-2]
	\arrow["{({r_s}_*)}"', from=2-2, to=2-3]
	\arrow[between={0.2}{0.8}, equals, from=0, to=1]
\end{tikzcd}\]
    The right triangle commutes by definition of a vu-retraction and the left square commutes by naturality. Since the diagonal $\delta$ is an injection, it follows that $f_p = g_p$, as desired.
\end{proof}

The converse direction is harder, and the proof will rely crucially on merging along a semi-flat diagram. The core ingredient is \Cref{prop:ultraretract-crux} which shows that being a vu-retract of a collection of points can be checked on the topological étalements above the sheaves in such a diagram.

\begin{nota}
    Let $\dsA$ be a collection of points of a vu-category $\X$.
    Given a vu-sheaf $Y \in\Sh(\X)$ we denote its total space by $\Y\to\X$, and we denote by $\dsA_Y$ the collection of points of $\Y$ that lie above a point in $\dsA$, \ie 
    $\dsA_Y$ consists of the pairs $(a, y)$ with $a \in \dsA$ and $y \in Y(a)$.
\end{nota}

\begin{lem}\label{prop:separating-slice}
     If $\dsA$ is a separating collection of points of $\X$, then $\dsA_Y$ is a separating collection of points in $\Y$.
\end{lem}
\begin{proof}
    Let $f: A \to B$ be a map of sheaves $\Sh(\Y)$ that is an isomorphism on the points of $\dsA_Y$, \ie for any $a\in\dsA$ and $y\in Y_a$ the map $f_{(a,y)} : A_{(a,y)}\to B_{(a,y)}$ is a bijection.
    We have $\Sh(\X)/Y \cong \Sh(\Y)$ by \Cref{thm:etale}\ref{prop:slice-etale}, this allows us to see $f$ as a map in $\Sh(\X)$. Concretely, we have the following diagram of étale vu-functors.
    \[\begin{tikzcd}
        \A && \B \\
        & \Y \\
        & \X
        \arrow["f", from=1-1, to=1-3]
        \arrow["{\pi_A}"', from=1-1, to=2-2]
        \arrow["{\pi_B}", from=1-3, to=2-2]
        \arrow["{\pi_Y}", from=2-2, to=3-2]
    \end{tikzcd}\]
    As $\dsA$ is separating in $\X$, it is enough to show that $f: \A\to\B$, seen as a map of sheaves above $\X$, is a bijection on the points of $\dsA$.
    But for any point $a\in\dsA$, the map on the stalks at $a$ is given by 
    \[\textstyle\sum_{y\in Y(a)}A_{(a,y)}\to\sum_{y\in Y(a)}B_{(a,y)}\]
    that is the sum of the maps $(f_{(a,y)} : A_{(a,y)}\to B_{(a,y)})_{y\in Y(a)}$ that are all bijections by assumption.
\end{proof}

\begin{lem}\label{lem:subcl-etalement}
    Let $Y\in\Sh(\X)$ and $(p,\xi)\in\Y$ such that $p \in \subcl_\X(\dsA)$, then $(p,\xi) \in \subcl_{\Y}(\dsA_Y)$.
\end{lem}
\begin{proof}
    By assumption, there exists a vu-retraction $i : p \vuto{\sigma}a_s$ and $(r_s : a_s \vuto{\tau_s}p)_{s:\sigma}$ in $\X$ with the $(a_s)$ in $\dsA$. 
    We denote $i_*(\xi) \in \int_{s:\sigma} Y_{a_s}$ by $(\zeta_s)_{s : \sigma}$, so that the pair $(i,\xi)$ gives a vu-arrow $(p,\xi) \vuto{\sigma} (a_s, \zeta_s)$ in $\Y$.
    Similarly, for $s:\sigma$, we denote ${r_s}_*(\zeta_s) \in \int_{t:\tau_s} Y_p$ by $(\omega_{s,t})_{t:\tau_s}$ so that the pair $(r_s, \zeta_s)$ gives a vu-arrow $(a_s, \zeta_s) \vuto{\tau_s} (p,\omega_{s,t})$ in $\Y$.
    Since $r_s \circ_\sigma i$ is the diagonal, we see that 
    \[
        ((\omega_{s,t})_{t:\tau_s})_{s:\sigma} = ({r_s}_*(\zeta_s))_{s:\sigma} =  (r_s \circ_\sigma i)_*(\xi) = ((\xi)_{t:\tau_s})_{s:\sigma},
    \]
    so that $(r_s, \zeta_s)$ is actually a vu-arrow $(a_s, \zeta_s) \vuto{\tau_s} (p,\xi)$.
    It only remains to see that the composite 
    \[(p,\xi)\vuto{\sigma}((a_s,\zeta_s)\vuto{\tau_s}(p,\xi))\]
    in $\Y$ is the diagonal vu-arrow, and this follows from the unique lifting property of the étale vu-functor $\Y\to\X$.
\end{proof}

\begin{prop}\label{prop:ultraretract-crux}
    Let $[-] : \dsD\to\Sh(\X)$ be a semi-flat diagram separating vu-arrows, $\dsA$ a collection of points of $\X$ and $p$ an object of $\X$.
    The following are equivalent:
    \begin{enumerate}
        \item\label{prop:ultraretract-crux-i}
        $p$ belongs to $\subcl(\dsA)$ in $\X$; 
        
        \item\label{prop:ultraretract-crux-ii}
        for each $\xi \in [D]_p$, the pair $(p, \xi)$ belongs to $\subcl(\dsA_{[D]})$ in $\Tref{\et{[D]}}$;

        \item\label{prop:ultraretract-crux-iii}
        for each pair $\ell = (D_\ell,\xi_\ell)\in\eM_{\dsD}(p)$ (see \Cref{nota:em}), there exists $i_\ell : p \vuto{\sigma_\ell}a_{\ell,s}$ and $(r_{\ell,s} : a_{\ell,s} \vuto{\tau_{\ell,s}}p)_{s:\sigma_\ell}$ in $\X$ with the $(a_{\ell,s})$ in $\dsA$, such that the composite $r_{\ell, s} \circ_{\sigma_\ell} i_\ell : p \vuto{\sigma_\ell.\tau_{\ell,s}}p$ sends $\xi_\ell$ to $(\xi_\ell)_{\sigma_\ell\cdot\tau_{\ell,s}}$.
    \end{enumerate}
\end{prop}

\begin{proof}\leavevmode
    \begin{enumerate}
        \item[\ref{prop:ultraretract-crux-i} $\Rightarrow$ \ref{prop:ultraretract-crux-ii}]
        Follows from \Cref{lem:subcl-etalement} and \Cref{rmk:subclosure-Tref}.   

        %

        \item[\ref{prop:ultraretract-crux-ii} $\Rightarrow$ \ref{prop:ultraretract-crux-iii}]
        We fix $\ell = (D,\xi) \in \eM_{\dsD}(p)$. By assumption,
        there exists a family $(a_s, \zeta_s)_{s : \sigma}$ in $\dsA_{[D_\ell]}$ such that 
        \[(p, \xi) \vuleq[{\Tref{\et{[D]}}}]{\sigma} (a_s, \zeta_s) \qquad \text{and} \qquad (a_s, \zeta_s) \vuleq[{\Tref{\et{[D]}}}]{} (p,\xi)  \text{, for $s : \sigma$ }.\]
        Given the definition of the topological reflection, we can find a vu-arrow $(p, \xi)\vuto{\sigma}(a_s, \zeta_s)$ (actually, to some thickening of $(a_s, \zeta_s)_{s : \sigma}$, but up to a re-indexing of the $(a_s,\zeta_s)$ we can suppose this thickening trivial) as well as vu-arrows $(a_s, \zeta_s) \vuto{\tau_s} (p,\xi)$.
        By construction of $\et{[D]}$, we obtain a vu-arrow $i : p \vuto{\sigma}a_s$ in $\X$ such that $i_*(\xi) = (\zeta_s)_{s : \sigma}$, and vu-arrows $(r_s : a_s \vuto{\tau_s}p)$ such that ${r_s}_*(\zeta_s) = (\xi)_{\tau_s}$ for $\sigma$-all $s$.
        It follows that
        \[(r_s\circ_\sigma i)_*(\xi) = ({r_s}_*(\zeta_s))_{s:\sigma} = ((\xi)_{t:\tau_s})_{s:\sigma}\]
        and therefore the condition \ref{prop:ultraretract-crux-iii} holds.
        
        \item[\ref{prop:ultraretract-crux-iii} $\Rightarrow$ \ref{prop:ultraretract-crux-i}]
        We consider an initial ultrafilter $\lambda$ on $\eM_{\dsD}(p)$ (see \Cref{def:lambda}), \ie such that, for each $\ell_0$, the set of $\ell$'s for which there exist $\gamma : \ell \to \ell_0$ in $\dsD$ is $\lambda$-large.

    
        Let $\sigma \coloneqq \lambda\cdot\sigma_\ell$, we merge all the maps $i_\ell : p \vuto{\sigma_\ell} a_{\ell,s}$ into a single $\sigma$-arrow,
        \[i : p\vuto[\delta]{\lambda}(p\vuto[i_\ell]{\sigma_\ell}{a_{\ell,s}})\]
        so that the codomain of $i : p\vuto{\sigma}a_{\ell,s}$ is the $\sigma$-family $(a_{\ell,s})$ of $\dsA$, and we collect all the $(r_{\ell,s})$ into a single $\sigma$-family of arrows, so that it is enough to show that the composite $r_{\ell,s} \circ_\sigma i$ is the diagonal $\delta : p \vuto{\sigma\cdot\tau_{\ell,s}}p$.
        
        Let $\tau_\ell \coloneqq (s : \sigma_\ell)\cdot\tau_{\ell,s}$, we consider the $\tau_\ell$ vu-arrows $(j_\ell \coloneqq r_{\ell, s} \circ_{\sigma_\ell} i_\ell)$ so that $r_{\ell,s} \circ_\sigma i = j_\ell \circ_\lambda \delta$, 
        \[p \vuto[i]{\sigma} (a_{\ell, s} \vuto[r_{\ell,s}]{\tau_{\ell,s}} p) = p \vuto[\delta]{\lambda} (p \vuto[i_\ell]{\sigma_\ell} (a_{\ell, s} \vuto[r_{\ell,s}]{\tau_{\ell,s}} p)) = p \vuto[\delta]{\lambda} (p \vuto[j_\ell]{\tau_\ell} p).\]
        
        Since the $[-] : \dsD\to\Sh(\X)$ separates vu-arrows, it suffices to check that each sheaf in $\dsD$ sends $j_\ell \circ_\lambda \delta $ and the diagonal to the same map of sets, that is, that for any $\ell_0 = (D_0,\xi_0) \in \eM_\dsD(p)$, the ultrafamily $(j_\ell \circ_\lambda \delta)_*(\xi_0) = \left({j_\ell}_*(\xi_{0})\right)_{\ell : \lambda}$ is the constant ultrafamily $((\xi_0)_{\tau_\ell})_{\ell:\lambda}$.
        
        We thus fix $\ell_0 = (D_0,\xi_0)$ in $\eM_\dsD(p)$, 
        and we show that ${j_\ell}_*(\xi_0)$ is the constant $\tau_\ell$-family at $\xi_0$ for all $\ell : \lambda$.
        By assumption on $\lambda$, we can assume that for all $\ell : \lambda$ there exists $\gamma : \ell \to \ell_0$, \ie $\gamma : D_\ell \to D_{0}$ such that $[\gamma]_p(\xi_\ell) = \xi_{0}$, and we consider the following naturality square of $j_\ell$, which shows ${j_\ell}_*([\gamma]_p(\xi_\ell)) = ([\gamma]_p)_{\tau_\ell} ({j_\ell}_*(\xi_{\ell}))$.
        \[\begin{tikzcd}
        	{\xi_\ell} & {[D_\ell]_p} & {\int_{\tau_\ell} [D_\ell]_p} & {} \\
        	{\xi_{0} = [\gamma]_p(\xi_\ell)} & {[D_{0}]_p} & {\int_{\tau_\ell}[D_{0}]_p}
        	\arrow["\in"{description}, draw=none, from=1-1, to=1-2]
        	\arrow[dashed, maps to, from=1-1, to=2-1]
        	\arrow["{{j_\ell}_*}", from=1-2, to=1-3]
        	\arrow["{[\gamma]_p}"', from=1-2, to=2-2]
        	\arrow["{([\gamma]_p)_{\tau_\ell}}", from=1-3, to=2-3]
        	\arrow["\in"{description}, draw=none, from=2-1, to=2-2]
        	\arrow["{{j_\ell}_*}"', from=2-2, to=2-3]
        \end{tikzcd}\]
        We conclude: 
        \begin{align*}
            {j_\ell}_*(\xi_{0}) &= {j_\ell}_*([\gamma]_p(\xi_\ell)) & \\
            &= ([\gamma]_p)_{\tau_\ell} ({j_\ell}_*(\xi_{\ell})) &\text{naturality of $j_\ell$} \\
            &= ([\gamma]_p)_{\tau_\ell} ((\xi_\ell)_{\tau_\ell}) & \text{assumption of \ref{prop:ultraretract-crux-iii}}\\
            &= ([\gamma]_p (\xi_\ell))_{\tau_\ell} & \\
            &= (\xi_{0})_{\tau_\ell}. &  \qedhere   
        \end{align*} 

    \end{enumerate}
\end{proof}

\begin{cor}\label{cor:ultraretract}
    Let $\X\subseteq\pt(\E)$ be a full sub-vu-category of $\pt(\E)$, and $\dsA$ a collection of points of $\X$. Then $\subcl_\X(\dsA) = \X$ if and only if $\subcl_{\Tref{\Y}}(\dsA_{Y}) = \Tref{\Y}$ for any $Y = \ev(E)\in\Sh(\X)$, $E\in\E$.
\end{cor}
\begin{proof}
    Follows from \Cref{prop:ultraretract-crux} above and \Cref{prop:semi-flat-exist}.
\end{proof}

\begin{Thm}\label{prop:subcl-sep-categorified}
    Let $\X\subseteq\pt(\E)$ be a full sub-vu-category of $\pt(\E)$, and $\dsA$ a collection of points of $\X$. Then $\dsA$ is separating if and only if $\subcl_\X(\dsA) = \X$. 
\end{Thm}
\begin{proof}
    It is a sufficient condition by \Cref{vu-retract-direction-facile}, we now show that it is necessary, we suppose $\dsA$ separating and show that $\subcl_X(\dsA) = \X$. By \Cref{cor:ultraretract} and \Cref{prop:subcl-sep-topological-case}, it is sufficient to show that $\dsA_{Y}$ is separating in $\Tref{\Y}$ for any $Y = \ev(E)$, $E\in\E$. But for each $E$, the collection of points $\dsA_{Y}$ is separating in $\Y$ by \Cref{prop:separating-slice}, and so, a fortiori, separating in $\Tref{\Y}$.
\end{proof}


\subsection{Geometric surjections}\label{sec:surj}

The present section can be understood as a ``relativisation'' of the previous one, where we define a class of maps based on the notion of subclosure and see that it corresponds to the all-important class of surjective geometric morphisms of toposes.
We recall that a morphism of locales $f : X\to Y$ is surjective if $f^* : \Open(Y)\to\Open(X)$ preserves the order strictly, similarly a geometric morphism is surjective if its inverse image is conservative. However, a continuous map of topological spaces $f: X\to Y$ can induce a surjection of locales without being surjective on points (for example the sobrification of a non-sober space), but taking the subclosure of the image fixes this discrepancy: a continuous map of topological spaces induces a surjection of locales if and only if $Y = \subcl(f(X))$. This leads to the following notion of \emph{vu-subdense functor}.

\begin{defn}\label{def:vu-subdense}
    A vu-functor $F : \X \to \Y$ is \emph{vu-subdense} if the subclosure of its image $F(\X) \subseteq \Y$ is the whole of $\Y$,
    \[\subcl_\Y(F(X)) = \Y,\]
    \ie if any point of $\Y$ is a vu-retract of an ultrafamily of points in the image of $F$.
\end{defn}

For topological spaces, \Cref{prop:subclosure-topology} lets us rephrase vu-subdensity.

\begin{prop}\label{prop:vu-subdense-topological}
    Let $f : X \to Y$ be a functor of topological spaces (or equivalently, essentially small taut vu-posets), the following assertions are equivalent:
    \begin{enumerate}
        \item \label{prop:vu-subdense-topological-i}%
        $f$ is vu-subdense; 
        \item \label{prop:vu-subdense-topological-ii}%
        $f^{-1} : \mathcal{P}(Y) \to \mathcal{P}(X)$ takes non-empty locally closed sets to non-empty locally closed sets; 
        \item \label{prop:vu-subdense-topological-iii}%
        $f^* : \Open(Y) \to \Open(X)$ preserves the order \emph{strictly}, that is $f$ induces a surjection of locales.
    \end{enumerate}
\end{prop}
\begin{proof}\leavevmode
    \begin{enumerate}
        \item[\ref{prop:vu-subdense-topological-i}$\Rightarrow$\ref{prop:vu-subdense-topological-ii}]%
        If $A$ is a non-empty locally closed set of $Y$, then $A$ must meet $f(X)$ non-trivially, that is, $f^{-1}(A)$ is non-empty, and it is clearly locally closed since $f$ is continuous.  
        \item[\ref{prop:vu-subdense-topological-ii}$\Rightarrow$\ref{prop:vu-subdense-topological-iii}]%
        If $U \subsetneq V$ in $\Open(Y)$, then $(Y \setminus U) \cap V$ is a non-empty locally closed set, hence so is its image under $f^{-1}$. That is, $f^*(U) \subsetneq f^*(V)$.
        \item[\ref{prop:vu-subdense-topological-iii}$\Rightarrow$\ref{prop:vu-subdense-topological-i}]%
        This is a reformulation of \Cref{prop:subcl-sep-topological-case}.\qedhere
    \end{enumerate}
\end{proof}


Perhaps unsurprisingly at this point, we go on to show that vu-subdense functors enjoy stability properties similar those we singled out for vu-full functors. The following properties are similar to the ones for vu-full functors from \Cref{prop:vufull}, except that vu-subdense functors are not stable under arbitrary pullbacks, but only under pullbacks along étale functors.

\begin{prop}\label{prop:subdense-stability}\leavevmode
    \begin{enumerate}
        \item\label{prop:subdense-stability:o} Let $F :\X\to \Y$ and $G:\Y\to\Z$ be vu-functors.
        If $F$ and $G$ are vu-subdense then so is $G\circ F$. If $G\circ F$ is vu-subdense then so is $G$. And if $G\circ F$ is vu-subdense and $G$ is fully faithful then so is $F$. 
        \item\label{prop:subdense-stability:i} Vu-subdense functors are stable under topological reflection.
        \item\label{prop:subdense-stability:ii} Vu-subdense functors are stable under pullback along étale functors, \ie if $F : \X \to \Y$ is vu-subdense, so is $F/A : F^*\A \to \A$ for each $A \in \Sh(\Y)$.
        \item\label{prop:subdense-stability:iii} Vu-subdense functors are stable under topological étalements (see \Cref{def:etalements}), \ie if $F : \X \to \Y$ is vu-subdense, so is $\Tref{F/A} : \Tref{F^*\A} \to \Tref{\A}$ for each $A \in \Sh(\Y)$.
    \end{enumerate}
        %
        %
\end{prop}
\begin{proof}
        \ref{prop:subdense-stability:o} is clear,
        \ref{prop:subdense-stability:i} is a consequence of \Cref{rmk:subclosure-Tref}, \ref{prop:subdense-stability:ii} follows from \Cref{lem:subcl-etalement}, 
        and \ref{prop:subdense-stability:iii} follows from the previous two.
\end{proof}

The main result of the previous section now takes a familiar form, analogous to \Cref{shc-vufull-cor}: assuming the codomain is a full sub-vu-category of one arising from the points of a topos, vu-subdense functors can be detected by their topological étalements.

\begin{prop}\label{prop:vu-subdense-slice}
    Let $\Y \subseteq \pt(\E)$ be a full sub-vu-category of $\pt(\E)$. A vu-functor $F : \X\to\Y$ is vu-subdense if and only if all its topological étalements above sheaves from $\E$ are vu-subdense, \ie  $\Tref{F/A} : \Tref{F^*\A} \to \Tref{\A}$ is vu-subdense for any $A = \ev(E) \in \Sh(\Y)$ with $E\in\E$.
\end{prop}
\begin{proof}
    The direct direction was just shown in \Cref{prop:subdense-stability}, and the converse follows from \Cref{cor:ultraretract}.
\end{proof}

The class of geometric surjections also enjoys similar stability conditions.

\begin{prop}\leavevmode\label{prop:surj-prop}
    \begin{enumerate}
        \item\label{prop:surj-prop:i} Geometric surjections are stable under localic reflection.
        \item\label{prop:surj-prop:ii} Geometric surjections morphisms are stable under pullback along étale geometric morphisms.
        \item\label{prop:surj-prop:iv} A geometric morphism $f : \F\to\E$ is surjective if and only if its localic slices (see \Cref{def:loc-slice}) are surjections of locales, \ie $\Lref{f/E}  : \Lref{\F/f^*E} \to\Lref{\E/E}$ 
        is a surjection of locales for any $E\in\E$.
    \end{enumerate}
\end{prop}
\begin{proof}
    The claim \ref{prop:surj-prop:i} is exactly \cite[A4.6.12]{johnstone:elephant},
    claim \ref{prop:surj-prop:ii} follows from inspection or \cite[C3.1.24]{johnstone:elephant} (take a discrete internal category there),
    and \ref{prop:surj-prop:iv} is a reformulation of \cite[VII.4, Lem.~3]{mac-lane-moerdijk}.
\end{proof}

Finally, we put together all the above results to get the following.
\begin{lem}\label{lem:surj-0dim}\leavevmode    
    \begin{enumerate}
        \item\label{lem:surj-0dim:i}\label{lem:surj-0dim:iii} A vu-functor between topological spaces is vu-subdense if and only if it induces a surjection of locales. In particular the sobrification $X\to\pt(\Open(X))$ is a vu-subdense embedding. 
        \item\label{lem:surj-0dim:ii} For any topos with enough points $\E$, the vu-functor $\Tref{\pt(\E)}\to\pt(\Lref{\E})$ is a vu-subdense embedding.
    \end{enumerate}
\end{lem}
\begin{proof}
    \ref{lem:surj-0dim:i} follows easily from \Cref{prop:vu-subdense-topological}. The vu-functor $\Tref{\pt(\E)}\to\pt(\Lref{\E})$ is fully faithful by \Cref{prop:rpt=ptr}, so it only remains to show that is it vu-subdense. For that, it is enough to notice that the points of $\E$ induce a separating class of points of $\Lref{\E}$ (because $\E\coverto\Lref{\E}$ is surjective and $\E$ is assumed to have enough points).
\end{proof}

\begin{Thm}\leavevmode\label{thm:surj}
    \begin{enumerate}
        \item Let $F : \X\to\Y$ be a vu-subdense vu-functor between bounded vu-categories, then $\Sh(F)$ is surjective. 
        \item Let $\X$ be a bounded vu-category, then $\eta_{\X} : \X\to\pt(\Sh(\X))$ is vu-subdense.
        \item Let $f: \F\to\E$ be a geometric surjection, then $\pt(f)$ is vu-subdense.
    \end{enumerate}
\end{Thm}
\begin{proof}\leavevmode
    \begin{enumerate}
        \item We show the hypothesis of \Cref{prop:surj-prop}\pref{prop:surj-prop:iv}. We take $B\in\Sh(\Y)$, by \Cref{prop:loc-slice} we have to show that 
        \[\shloc(\Tref{F^*\B}) \to \shloc(\Tref{\B})\] is a surjection of locales.
        
        By \Cref{prop:subdense-stability}\ref{prop:subdense-stability:iii} $\Tref{F^*\B} \to \Tref{\B}$ is vu-subdense, and we conclude by \Cref{lem:surj-0dim}\ref{lem:surj-0dim:iii}.

        \item We show the hypothesis of \Cref{prop:vu-subdense-slice}. We take $A\in\Sh(\X)$, by \Cref{prop:etalements}\ref{prop:etalements:i} we have to show that 
        \[\Tref{\A}\to\Tref{\pt(\Sh(\A))}\]
        is vu-subdense.
        

        We have
        \[\begin{tikzcd}
        	{\Tref{\A}} & {\Tref{\ptvu(\shvu(\A))}} \\
        	{\ptloc(\shloc(\Tref{\A})))} & {\ptloc(\Lref{\shvu(\A)})}
        	\arrow[from=1-1, to=1-2]
        	\arrow["{\text{vu-subd.}}"', hook, from=1-1, to=2-1]
        	\arrow[hook, from=1-2, to=2-2]
        	\arrow["{\cong}", from=2-1, to=2-2]
        \end{tikzcd}\]
        where $\ptloc(\shloc(\Tref{\A})))\cong \ptloc(\Lref{\shvu(\A)})$ by \Cref{circular-square}, 
        $\Tref{\ptvu(\shvu(\A))}\mono\ptloc(\Lref{\shvu(\A)})$ is an embedding by \Cref{prop:rpt=ptr},
        $\Tref{\A}\to\ptloc(\shloc(\Tref{\A})))$ is vu-subdense by \Cref{lem:surj-0dim}\ref{lem:surj-0dim:i},
        and we conclude by \Cref{prop:subdense-stability}\ref{prop:subdense-stability:o}.

        \item We will show that $\pt(f)$ satisfies \ref{prop:vu-subdense-slice}. We take $A\in\E$, by \Cref{prop:etalements}\ref{prop:etalements:ii} we have to show that
        \[\Tref{\pt(\F/f^*A)} \to \Tref{\pt(\E/A)}\]
        is vu-subdense.
        
        We have
        \[\begin{tikzcd}
        	{\Tref{\pt(\F/f^*A)}} & {\Tref{\pt(\E/A)}} \\
        	{\pt(\Lref{\F/f^*A})} & {\pt(\Lref{\E/A})}
        	\arrow[from=1-1, to=1-2]
        	\arrow["{\text{vu-subd.}}"', hook, from=1-1, to=2-1]
        	\arrow[hook, from=1-2, to=2-2]
        	\arrow["{\text{vu-subd.}}", from=2-1, to=2-2]
        \end{tikzcd}\]
        where $\Tref{\pt(\F/f^*A)}\mono\pt(\Lref{\F/f^*A})$ is a vu-subdense embedding by \Cref{lem:surj-0dim}\ref{lem:surj-0dim:ii},
        $\Tref{\pt(\E/A)}\mono\pt(\Lref{\E/A})$ is an embedding by \Cref{prop:rpt=ptr},
        $\pt(\Lref{\F/f^*A} )\to \pt(\Lref{\E/A})$ is vu-subdense by \Cref{prop:surj-prop}\ref{prop:surj-prop:iv} and \Cref{lem:surj-0dim},
        and we conclude by \Cref{prop:subdense-stability}\ref{prop:subdense-stability:o}.
        
        
        \qedhere
     \end{enumerate}
\end{proof}

Together with the reconstruction theorem (\Cref{thm:reconstruction}) we can now flesh out the relation between vu-subdense functors and surjective geometric morphisms, 
        %
        %
and fill the last line of the table (\ref{tableau}). 

\begin{cor}\label{table:surj}
    \
    
    \centerline{
    \begin{tabular}{r l||C C|C C|C C|}
    & & \multicolumn{2}{c|}{$f\squig\pt(f)$}&
        \multicolumn{2}{c|}{$\Sh(F)\squig F$}&
        \multicolumn{2}{c|}{$f\squig F$}\\
    \cline{3-8}
    geometric morphisms & vu-functors & $\Rightarrow$ & $\Leftarrow$ & $\Rightarrow$ & $\Leftarrow$ & $\Rightarrow$ & $\Leftarrow$ \\
    \cline{1-8}
    surjective & vu-subdense 
        & \checkmark & \checkmark 
        & ? & \checkmark
        & \checkmark & \checkmark
    \end{tabular}
    }

\end{cor}

\begin{rmk}\label{cor:vu-subdense:counterex}
    When the codomain of $F$ does not have enough sheaves, does $\Sh(F)$ being surjective entail that $F$ is vu-subdense? We expect that answering this question requires a better algebraic understanding of vu-retractions in non-sober vu-categories, which we defer to future work.
\end{rmk}

\printbibliography

\end{document}